\documentclass[12pt]{article}

\usepackage{amsfonts}
\usepackage{amssymb}
\usepackage{amsthm}
\usepackage{amsmath}
\usepackage{amscd}
\usepackage{mathrsfs}
\usepackage{enumerate}
\usepackage{ulem}
\usepackage{fontenc}
\usepackage[all]{xy}
\usepackage{array}
\usepackage[english]{babel}
\usepackage{float}
\usepackage{longtable}
\usepackage[ruled,vlined]{algorithm2e}
\SetKwFor{For}{for (}{) $\lbrace$}{$\rbrace$}

\def\CC {{\mathbb C}}     
\def\NN {{\mathbb N}}     
\def\PP {{\mathbb P}}     
\def\QQ {{\mathbb Q}}     
\def\RR {{\mathbb R}}     
\def\ZZ {{\mathbb Z}}     

\def\mc {\mathcal}
\def\mk {\mathfrak}

\def\ol  {\overline}

\def\tst {\Longleftrightarrow}

\newtheorem{theorem}{Theorem}[section]
\newtheorem{lemma}[theorem]{Lemma}

\newtheorem{prop}[theorem]{Proposition}

\newtheorem{coro}[theorem]{Corollary}
\newtheorem{rem}{Remark}[section]
\newtheorem{definition}{Definition}[section]

\newtheorem{example}{Example}[section]

\begin{document}

\title{On reductive subgroups of reductive groups\\ having invariants in almost all representations}
\author{Valdemar Tsanov, Yana Staneva \footnote{Both authors are supported by DFG grant SFB-TR191, ``Symplectic Structures in Geometry, Algebra and Dynamics''.}}

\maketitle

\begin{abstract}
Let $G$ and $\tilde G$ be connected complex reductive Lie groups, $G$ semisimple. Let $\Lambda^+$ be the monoid of dominant weights for a positive root system $\Delta^+$, and let $l(w)$ be the length of a Weyl group element $w$. Let $V_\lambda$ denote an irreducible $G$-module of highest weight $\lambda\in\Lambda^+$. For any closed embedding $\iota:\tilde G\subset G$, we consider\\

Property (A): $\quad\forall\lambda\in\Lambda^+,\exists q\in\mathbb{N}$ such that $V_{q\lambda}^{\tilde G}\ne0$.\\

A necessary condition for (A) is for $G$ to have no simple factors to which $G$ projects surjectively. We show that this condition is sufficient if $\tilde G$ is of type ${\bf A}_1$ or ${\bf E}_8$.

We define and study an integral invariant of a root system, $\ell_G=\min\{\ell^\lambda:\lambda\in\Lambda^+\setminus\{0\}\}$, where $\ell^\lambda=\min\{l(w):w\lambda\notin{\rm Cone}(\Delta^+)\}$. We derive the following sufficient condition for (A), independent of $\iota$: 
$$
\ell_G - \#\tilde\Delta^+ > 0 \;\Longrightarrow\; (A).
$$
We compute $\ell_G$ and related data for all simple $G$, except ${\bf E}_8$, where we obtain lower and upper bounds. We consider a stronger property (A-$k$) defined in terms of Geometric Invariant Theory, related to extreme values of codimensions of unstable loci, and derive a sufficient condition in the form $\ell_G - \#\tilde\Delta^+ > k$. The invariant $\ell_G$ proves too week to handle $G=SL_n$ and we employ a companion $\ell_G^{\rm sd}$ to infer (A-$k$) for a larger class of subgroups. We derive corollaries on Mori-theoretic properties of GIT-quotients.
\end{abstract}


{\small
\tableofcontents
}

\section{Introduction}\label{intro}

We study embeddings $\iota:\tilde G\subset G$ of connected reductive complex Lie groups, specifically, structural properties of the root systems with applications to Geometric Invariant Theory and branching laws for representations. By the Cartan-Weyl theorem, if $H\subset B\subset G$ is a pair of Cartan and Borel subgroups of $G$, the irreducible representations of $G$ are parametrized by the set $\Lambda^+$ of $B$-dominant weights in the weight lattice $\Lambda$ of $H$. We denote by $V_\lambda$ an exemplary irreducible $G$-module with highest weight $\lambda\in\Lambda^+$ and define
$$
LR_0(\iota)=\{\lambda\in\Lambda^+:V_\lambda^{\tilde G}\ne 0\} \;,\;\; LR(\iota)=\{(\tilde\lambda,\lambda)\in\tilde\Lambda^+\times\Lambda^+: (\tilde V^*_{\tilde\lambda}\otimes V_\lambda)^{\tilde G}\ne 0 \} = LR_0({\rm id}\times \iota) \;.
$$
These are known to be finitely generated monoids by a theorem of Brion and Knop, cf. \cite{Elashvili-92}. Since $(\tilde V^*_{\tilde\lambda}\otimes V_\lambda)^{\tilde G}\cong {\rm Hom}_{\tilde G}(\tilde V_{\tilde\lambda},V_\lambda)$, the monoid $LR(\iota)$ describes the irreducible representations of $\tilde G$ occurring in the restriction of a given irreducible representation of $G$.  The rational polyhedral convex cones spanned by these monoids in the respective Weyl chambers over the reals are denoted by
$$
\mc{LR}_0(\iota)={\rm Cone}(LR_0(\iota))\subset\Lambda_\RR^+\;,\;\;\mc{LR}(\iota)={\rm Cone}(LR(\iota))=\mc{LR}_0({\rm id}\times\iota)\subset \tilde\Lambda_\RR^+\times\Lambda_\RR^+\;,
$$
where $\Lambda_\RR^+={\rm Cone}(\Lambda^+)\subset\Lambda_\RR=\Lambda\otimes_\ZZ\RR$. $\mc{LR}(\iota)$ is known as the generalized Littlewood-Richardson cone, or also eigencone, for the embedding $\iota$. We call $\mc{LR}_0$ the $0$-eigencone. These cones admit several important interpretations in various contexts. We recall further on the interpretation in terms of Geometric Invariant Theory (GIT) applied to subgroup actions on flag varieties, or alternatively, in terms of momentum maps and symplectic quotients applied to projections of coadjoint orbits of compact Lie algebras to subalgebras. 

A description of $\mc{LR}(\iota)$ by an irredundant set of inequalities was obtained by Ressayre, \cite{Ress-Eigencone}. The method is indeed based on the identification $\mc{LR}(\iota)=\mc{LR}_0({\rm id}\times\iota)$ and a description of the more general $\mc{LR}_0$ for arbitrary embeddings. We observe at this point a difference in behaviour between eigencones and $0$-eigencones, which, as we show, characterizes the former among the latter, under some assumptions on the groups. Namely, the list of inequalities for $\mc{LR}(\iota)$ is never trivial, as $\mc{LR}(\iota)\cap(\tilde\Lambda_\RR^+\times\{0\})=\{(0,0)\}$, while certain $0$-eigencones are known to fill the Weyl chamber $\Lambda_\RR^+$. This property has several interesting interpretations, and its presence reduces the difficulty of some invariant theoretic problems. For instance, the saturation problem for $LR_0(\iota)$ is reduced to finding the saturation coefficients $q$ for the fundamental weights of $G$ and taking the least common multiple. Thus, characterizing such embeddings presents an interesting problem.\\

\noindent{\bf Property (A)} is said to hold for an embedding $\iota:\tilde G\subset G$, if $\mc{LR}_0(\iota)=\Lambda_\RR^+$. Equivalently, there exists $q\in\NN$ such that $V_{q\lambda}^{\tilde G}\ne 0$ for all $\lambda\in\Lambda^+$.\\

Our initial motivation came from the observation that this property holds for principal $SL_2$-subgroups of simple groups of rank greater than 1, cf. \cite{Sepp-Tsan-Princ}. Consequently, property (A) characterizes the eigencones among the $0$-eigencones, for principal $SL_2$-subgroups of simply connected semisimple groups. We are now able to show that this holds for all $SL_2$-subgroups. This fact is maybe known to experts, or not very surprising, though we are unaware of a clear formulation in the literature. The case is considered by Berenstein-Sjamaar, \cite{Beren-Sjam}, where a specific but redundant list of inequalities for $\mc{LR}_0(\iota)$ for $SL_2$-subgroups is given. We obtain a single inequality for each $SL_2$-factor of $G$ to which $\tilde G$ projects nontrivially. The presence of such a factor allows, upon writing $G=SL_2\times G_1$, to interpret the $0$-eigencone as an eigencone, $\mc{LR}_0(\iota)=\mc{LR}({\rm id}\times\iota_1)$. More importantly, we are interested in a proof using only general structural properties of the groups, and valid for a larger class of subgroups.

In a nutshell, the idea is to combine the inequalities describing $\mc{LR}_0(\iota)$ into a single inequality $\ell_G>\dim \tilde G/\tilde B$ implying $\mc{LR}_0(\iota)=\Lambda_\RR^+$. Here $\ell_G$ is a suitable numerical invariant of $G$ defined below. The resulting sufficient condition for (A) is not necessary, but is independent of $\iota$, derived only from the isomorphism types of the Lie algebras $\mk g$ and $\tilde{\mk g}$. While the derivation of this sufficient condition is straightforward, given the known results, we contribute a precise combinatorial definition for $\ell_{G}$, and a companion $\ell_G^{\rm sd}$, and compute their values for all types of semisimple groups, except ${\bf E}_8$, for which we only obtain lower and upper bounds. This allows us to derive several corollaries. In particular, our condition admits some optimal properties, as it becomes necessary for some cases of $\tilde G$. Our results indicate that property (A) holds for all ``sufficiently small'' subgroups of $G$, and we are able to prove this, save for the case where $G$ has simple factors of type ${\bf A}_n$, where we need a supplementary assumption.

We apply a GIT method developed in \cite{Sepp-Tsan-2020}, allowing to consider an even stronger property (A-k), defined in the next section, imposing an analogue of a lower bound $k\leq\dim V_{q\lambda}^{\tilde G}$ for some $q\in \NN$. The case $k=2$ is of special importance, as it implies strong Mori-theoretic properties for the GIT quotients of the flag variety $G/B$ by $\tilde G$, yielding information on the total invariant ring ${\rm Cox}(G/B)^{\tilde G}\cong\oplus_\lambda V_\lambda^{\tilde G}$. Our sufficient conditions for (A-k) are formulated in Theorem \ref{Theo tildeDelta ellK A}. As a corollary, we show that all proper ${\bf E}_8$-subgroups of simple groups satisfy property (A-2). Consequently, for $\tilde{\mk g}\cong\mk e_8$ property (A-2) characterizes the eigencones among the $0$-eigencones.

Let us now define the aforementioned invariant of $G$. Let $\Delta=\Delta^+\sqcup\Delta^-$ be the root system of $G$ split into positive and negative part by the choice of $\Lambda^+$. Let $W$ be the Weyl group and $l(w)=\Delta^+\cap w^{-1}\Delta^-$ denote the length of $w\in W$. Let $w_0\in W$ denote the longest element, $l(w_0)=\#\Delta^+$. For $\lambda\in\Lambda^+$, the weight $\lambda^*=-w_0\lambda$ is called the dual to $\lambda$, as $V_\lambda^*\cong V_{\lambda^*}$. We denote by $\Lambda^+_{\rm sd}=\{\lambda\in\Lambda^+:\lambda=\lambda^*\}$ the set of self-dual dominant weights. We set:
\begin{gather}\label{For Def ellK}
\begin{array}{l}
\ell_G=\ell_\Delta=\min\{l(w): w\in W, w\Lambda^+ \nsubseteq {\rm Cone}(\Delta^+)\}\;, \\
\ell^{\rm sd}_G=\ell^{\rm sd}_\Delta=\min\{l(w): w\in W, w\Lambda_{\rm sd}^+ \nsubseteq {\rm Cone}(\Delta^+)\}\;.
\end{array}
\end{gather}
It is easy to show that the value of the invariant $\ell_\Delta$ for a given root system is $\ell_\Delta=\min\{\ell_{\Delta_1},...,\ell_{\Delta_m}\}$, where $\Delta=\Delta_1\sqcup\dots\sqcup\Delta_m$ is the decomposition into simple components. The same holds for $\ell^{\rm sd}_\Delta$.

The following is a particular case of Theorem \ref{Theo tildeDelta ellK A}, which handles the more general property (A-$k$), for which the sufficient condition has the form $\ell_\Delta-\#\tilde\Delta^+\geq k$.

\begin{theorem}\label{Theo Intro Main}
Let $G$ and $\tilde G$ be complex reductive groups, $G$ semisimple. Then the following hold.
\begin{enumerate}
\item[{\rm (i)}] If $\ell_\Delta>\#\tilde\Delta^+$, then property (A) holds for any embedding $\iota:\tilde G\subset G$.
\item[{\rm (ii)}] If $\tilde w_0=-1$ and $\ell_\Delta^{\rm sd}>\#\tilde\Delta^+$, then property (A) holds for any embedding $\iota:\tilde G\subset G$.
\end{enumerate}
\end{theorem}

In view of the above result it is of interest to know the values of the two invariants defined above. These are given in the theorem below for simple root systems, except ${\bf E}_8$, for which we only obtain bounds. The value for an arbitrary root system is obtained as the minimum of the values for the simple constituents, as shown in Corollary \ref{Coro ellDeltaDual}. Let us note that $\ell_\Delta$ and $\ell_\Delta^{\rm sd}$ coincide for trivial reasons, whenever $w_0=-1$, i.e. for types ${\bf B}_n$, ${\bf C}_n$, ${\bf D}_{2n}$, ${\bf E}_7$, ${\bf E}_8$, ${\bf F}_4$, ${\bf G}_2$.

\begin{theorem}\label{Theo ellDelta Classification}
The values of $\ell_\Delta$ and $\ell^{\rm sd}_\Delta$ for the simple root systems are the following:\\

\begin{tabular}{|l|c|c|}
\hline
Type of $\Delta$ & $\ell_\Delta$ & $\ell_\Delta^{\rm sd}$ \\
\hline
${\bf A}_n$ & $1$ & $\lfloor\frac{n+2}{2}\rfloor$ \\
\hline
${\bf B}_n$ & $n$ &  \\
\hline
${\bf C}_n$ & $n$ &  \\
\hline
${\bf D}_n$, $n\geq 4$ & $n-1$, for $n\ne5$ & $n-1$ \\
            & $3$, for $n=5$ &  \\
\hline
${\bf E}_6$ & $5$ & $9$ \\
\hline
${\bf E}_7$ & $10$ &  \\
\hline
${\bf E}_8$ & $7\leq\ell_{{\bf E}_8}\leq29$ &  \\
\hline
${\bf F}_4$ & $8$ & \\
\hline
${\bf G}_2$ & $3$ & \\
\hline
\end{tabular}
\end{theorem}

The proof is given in Section \ref{Sect Compute}, Theorems \ref{Theo Values ellDelta} and \ref{Theo Values ellsdDelta}. We use case by case analysis and combinatorics of Weyl groups. We derive detailed information on the values of the individual $\ell^\lambda=\min\{l(w):w\lambda\notin{\rm Cone}(\Delta^+)\}$ for various $\lambda\in\Lambda^+$, in particular the fundamental weights, where the value of $\ell_\Delta$ is attained.

Types ${\bf A}_1$ and ${\bf E}_8$ are exactly the simple types of $\tilde G$ for which the conditions of Theorem \ref{Theo Intro Main} apply for all proper embeddings into simple groups. We infer the following.

\begin{coro}
Suppose $\tilde G$ is of type ${\bf A}_1$ or ${\bf E}_8$ and $\iota:\tilde G\subset G$ is a closed embedding into a reductive $G$. Then property (A) holds for $\iota$ if and only if $G$ has no simple factors to which $\tilde G$ projects surjectively.
\end{coro}

The two cases are handled separately, in Theorems \ref{Theo SL2} and \ref{Theo E8 all movable}, where also stronger properties in each case are deduced.

\begin{rem}\label{Rem SLn}
Let us observe that type ${\bf A}_n$, i.e., the case $G=SL_{n+1}$, exhibits a unique behaviour of the invariants $\ell_\Delta$ and $\ell_\Delta^{\rm sd}$. This case is in fact the main reason for the introduction of $\ell_\Delta^{\rm sd}$, since the value $\ell_{{\bf A}_n}=1$ makes part (i) of Theorem \ref{Theo tildeDelta ellK A} applicable only if $\tilde G$ is a torus. Our initial goal was to handle and generalize the case of $\tilde G= SL_2$. Part (i) implies property (A) for $SL_2$-subgroups in $G$ of all simple types except ${\bf A}_n$.

Let us also note that there is room for refinement and modification of the proposed methods, taking, as it may, further properties of the embedding into account. For instance, if the subgroup $\tilde G$ is assumed to be semisimple, then one may devise a third invariant of $\Delta$ by taking $\Lambda^+_{\rm sesi}$ to be the set of dominant weights identified under the Killing form with semisimple elements of $\mk g$ contained in proper semisimple subalgebras. Put $\ell_\Delta^{\rm sesi}=\min\{l(w):w\Lambda^+_{\rm sesi} \nsubseteq {\rm Cone}(\Delta^+)\}$. The proof of Theorem \ref{Theo tildeDelta ellK A} can be modified to show that the inequality $\ell^{\rm sesi}_\Delta>\#\tilde\Delta^+$ implies $\mc{LR}_0(\iota)=\Lambda_\RR^+$ for any embedding $\iota:\tilde G\subset G$. We have not computed the values of $\ell_\Delta^{\rm sesi}$, but for type ${\bf A}_n$ we observe that $\Lambda^+_{\rm sesi}$ does not contain any of the non-self-dual fundamental weights $\varpi_1,...,\varpi_{\lfloor n/2\rfloor}$. These are indeed the elements responsible for the low value of $\ell_\Delta$, as it appears from the proof of Theorem \ref{Theo ellDelta Classification}. It is thus possible that $\ell_\Delta^{\rm sesi}$ may be used instead of $\ell_\Delta^{\rm sd}$ to deduce property (A) for $SL_2$-subgroups.
\end{rem}

The article consists of three sections. In the next section we give our main definitions from GIT, introduce property (A-$k$) and formulate the main criterion for (A=$k$) in Theorem \ref{Theo tildeDelta ellK A}. We also derive corollaries from the criterion and the known values of $\ell_\Delta$ as given in Theorem \ref{Theo ellDelta Classification}. In Section \ref{Sect Proof of Main} we give a proof of Theorem \ref{Theo tildeDelta ellK A}. In Section \ref{Sect Compute} we study the invariants $\ell_\Delta$ and $\ell_\Delta^{\rm sd}$. We derive some general properties and compute their values as well as other related data. The values for classical groups are derived here. Most of the data for exceptional groups other than ${\bf E}_8$ is obtained with the use of a computer program available at \cite{Compute}.\\

\noindent{\bf Acknowledgment:} We thank Peter Heinzner for his support.

\section{GIT setting and main results}\label{Sect Settnmain}

The flag varieties of semisimple complex Lie groups, i.e., the quotient spaces by parabolic subgroups, constitute exactly the class of homogeneous projective varieties. From now on, unless otherwise specified, $G$ denotes a semisimple simply connected complex Lie group, $B\subset G$ a Borel subgroup and $H\subset B$ a Cartan subgroup. We denote by $X=G/B$ the complete flag variety of $G$. Then any flag variety of $G$ can be written as $G/P$ with a parabolic subgroup $B\subset P\subset G$.

The Picard group of $X$ is identified with the weight lattice $\Lambda$, by $\lambda\mapsto\mc L_\lambda=G\times_B\CC_{-\lambda}$. By the Borel-Weil theorem, the effective line bundles are exactly those defined by dominant weights and $H^0(X,\mc L_\lambda)\cong V_\lambda^*$ for $\lambda\in\Lambda^+$. The ample line bundles correspond to the strictly dominant weights, whose set is denoted by $\Lambda^{++}$. The image of $X$ under the map to $\PP(V_\lambda)$ defined by $\mc L_\lambda$ with $\lambda\in\Lambda^+$ is the $G$-orbit of the highest weight line and will be denoted by $X_\lambda$, so $X\to X_\lambda= G[v_\lambda]\subset \PP(V_\lambda)$. Let $P_\lambda=G_{[v_\lambda]}$ be the (parabolic) isotropy subgroup, so that $X_\lambda=G/P_\lambda$. The section ring $\mc R(\lambda)$ of $\mc L_\lambda$ and the Cox ring, or total coordinate ring, of $X$ are thus given by
$$
\mc R(\lambda)=\bigoplus\limits_{q=0}^\infty H^0(X,\mc L_\lambda^q)\cong
\bigoplus\limits_{q=0}^\infty V_\lambda^*\;,\;\; {\rm Cox}(X)=
\bigoplus\limits_{\mc L\in{\rm Pic}(X)}^\infty H^0(X,\mc L)\cong\bigoplus\limits_{\lambda\in\Lambda^+}^\infty V_\lambda \;.
$$

Given an embedding $\iota:\tilde G\to G$, we obtain a $\tilde G$-action on $X$, yielding homogeneous linear actions on each of the above rings. Let $\mc R_+(\lambda)$ denote the ideal in $\mc R(\lambda)$ spanned by the sections of positive powers of $\mc L_\lambda$. The common vanishing locus of the $\tilde G$-invariants in $\mc R_+(\lambda)$ is called the nullcone, or unstable locus, denoted by
$$
X^{us}_{\tilde G}(\lambda) = Z(\mc R_+(\lambda)^{\tilde G}) \;.
$$
Clearly, for a nonzero $\lambda\in\Lambda^+$, we have $\lambda\in\mc{LR}_0(\iota)$ if and only if $X^{us}_{\tilde G}(\lambda)\ne X$. Furthermore, the definition $X^{us}_{\tilde G}(\lambda)$ extends to $\lambda\in\Lambda_\QQ$, by consideration of powers, and to $\Lambda_\RR={\rm Pic}(X)_\RR$ by continuity. The elements $\lambda\in\Lambda_\RR^{++}$ (resp. $\Lambda_\RR^+$) with $X^{us}_{\tilde G}(\lambda)\ne X$ are called $\tilde G$-ample (resp. $\tilde G$-semiample). The $\tilde G$-ample elements form a cone in $\Lambda_\RR^{++}$, called the $\tilde G$-ample cone on $X$, and denoted by $C^{\tilde G}(X)$. It satisfies $C^{\tilde G}(X)=\mc{LR}_0(\iota)\cap\Lambda_\RR^{++}$ and, whenever it is nonempty, $\ol{C^{\tilde G}(X)}=\mc{LR}_0(\iota)$. The $\tilde G$-ample cone is partitioned into GIT-classes by the equivalence relation: $\lambda\sim\lambda'$ if and only if $X^{us}_{\tilde G}(\lambda)=X^{us}_{\tilde G}(\lambda')$. By a result of Ressayre, \cite{Ress-GITclasses}, the GIT-classes form a fan in $C^{\tilde G}(X)$. A concrete description of this fan for the case of the flag variety $X=G/B$ is given in \cite{Sepp-Tsan-2020}. 

In \cite{Sepp-Tsan-2020} it is shown that the codimension of the unstable locus varies by at most 1, between GIT-classes whose closures intersect. More specifically, we have the following. We denote
$$
{\rm cod}(\lambda)={\rm cod}_{\tilde G}(\lambda) = {\rm codim}_X X^{us}_{\tilde G}(\lambda) \quad{\rm for}\quad \lambda\in\Lambda^+ \;,
$$
and define, for $k\in\NN$,
$$
\mc C_k(\iota)=C_k^{\tilde G}(X) = \{\lambda\in\Lambda_\RR^{++}: {\rm cod}(\lambda)\geq k\} \;.
$$

For $k=1$ we obtain the $\tilde G$-ample cone, $\mc C_1(\iota)=C^{\tilde G}(X)$.

For $k=2$ we obtain the so-called $\tilde G$-movable cone, denoted ${\rm Mov}^{\tilde G}(X)=\mc C_2(\iota)$, which is also of special importance as explained in the next section.

It is shown in \cite{Sepp-Tsan-2020} that $\mc C_k$ is a rational polyhedral cone in the interior of the Weyl chamber $\Lambda_\RR^{++}$. Furthermore, $\mc C_{k+1}(\iota)$ is contained in the relative interior of $\mc C_k(\iota)$. We denote
$$
k(\iota)=\min\{{\rm cod}(\lambda):\lambda\in\Lambda^+\setminus\{0\}\}\;.
$$
We consider the following properties for embedding $\iota:\tilde G\subset G$:\\

\noindent{\bf Property (A-$k$)} is said to hold for $\iota$ if $\mc{LR}_0(\iota)=\ol{\mc C_k(\iota)}=\Lambda_\RR^+$. Equivalently, $k(\iota)\geq k$.\\

Our main observation is the following.

\begin{theorem}\label{Theo tildeDelta ellK A}
Let $G$ be a connected semisimple complex Lie group and let $\iota:\tilde G\subset G$ be a reductive closed connected complex subgroup. Then the following hold: 
\begin{enumerate}
\item[{\rm (i)}] $k(\iota)\geq \ell_\Delta-\#\tilde\Delta^+$. In particular, if the inequality $\#\tilde\Delta^+<\ell_\Delta$ holds, then
$$
\mc{LR}_0(\iota)=\ol{\mc C_{\ell_\Delta-\#\tilde\Delta^+}(\iota)}=\Lambda_\RR^+
$$
and property (A-$k$) holds for $1\leq k\leq \ell_\Delta^{\rm sd}-\#\tilde\Delta^+$.
\item[{\rm (ii)}] Suppose in addition $\tilde w_0=-1$. Then $k(\iota)\geq \ell^{\rm sd}_\Delta-\#\tilde\Delta^+$. In particular, if $\#\tilde\Delta^+<\ell^{\rm sd}_\Delta$ holds, then
$$
\mc{LR}_0(\iota)=\ol{\mc C_{\ell_\Delta^{\rm sd}-\#\tilde\Delta^+}(\iota)}=\Lambda_\RR^+
$$
and property (A-$k$) holds for $1\leq k\leq \ell_\Delta^{\rm sd}-\#\tilde\Delta^+$.

\end{enumerate}
\end{theorem}

The proof is given in Section \ref{Sect Proof of Main}. It is based on a result of \cite{Sepp-Tsan-2020} cited here as Theorem \ref{Theo SeppTsan Xuslambda cod}, giving a formula for the unstable locus $X^{us}_{\tilde G}(\lambda)$ and its codimension. Fundamentally, all considerations are enabled by the Hilbert-Mumford criterion and its explicit application for flag varieties, which naturally relates to Schubert cell decompositions.\\

In the rest of this section we recall some GIT notions and results related to the condition (A-2), and derive corollaries from the above theorem and the knowledge of the values of $\ell_\Delta$ and $\ell_\Delta^{\rm sd}$ given in Theorem \ref{Theo ellDelta Classification}.

\begin{rem}
The above theorem provides a sufficient condition for property (A-$k$), independent of the embedding $\iota$, referring only to the respective root systems of the two groups. The condition can be sharpened if the embedding is taken into account, as can be seen from the proof, Lemma \ref{Lemma ElCk} in particular. In this generality, however, the condition is in a sense optimal. Indeed, for part (i), we observe that taking $\tilde G=H$ to be a Cartan subgroup of $G$, we obtain
$$
k(H\subset G) = \ell_\Delta \;.
$$
For part (ii), we observe that for $SL_2$-subgroups, the equality $\#\tilde\Delta^+=\ell^{\rm sd}_\Delta$ is equivalent to the presence of a simple factor of $G$ of rank $1$. By choosing an $SL_2$-subgroup $\iota:\tilde G\subset G$ projecting nontrivially to this factor, we obtain $\mc{LR}_0(\iota)\ne\Lambda_\RR^+$, as remarked in the Introduction.
\end{rem}

\subsection{GIT-quotients and the $\tilde G$-movable cone on $G/B$}\label{Sect GIT quotients}

Here we recall come definitions from GIT as well as a result showing the significance of property (A-2), renamed here as property (M). We also derive a corollary from our theorem concerning this property.

For $\lambda\in\mc C_1(\iota)$, the semistable locus $X^{ss}_{\tilde G}(\lambda)=X\setminus X^{us}_{\tilde G}(\lambda)$ is nonempty. The GIT-quotient $Y_\lambda=X^{ss}_{\tilde G}(\lambda)//G$ is defined by Hilbert's relation, where two points are equivalent if the closures of their $\tilde G$-orbits intersect in the semistable locus. The quotient is clearly independent of the choice of $\lambda$ within its GIT-class.

A GIT-class $C\subset \mc C_1(\iota)$ is called a chamber, or $\tilde G$-chamber, if all semistable $\tilde G$-orbits are infinitesimally free, or equivalently, if $X^{us}_{\tilde G}(\lambda)$ contains all points with positive dimensional $\tilde G$-isotropy group. The chamber-quotients are geometric quotients. It is shown in \cite{Sepp-Tsan-2020} that the $\tilde G$-chambers on $G/B$ are exactly the GIT-classes open in $\Lambda_\RR$, whenever they exist.

The Picard group of a GIT-quotient is naturally related to the Picard group of the original variety, cf. \cite{KKV}. This relation becomes stronger when the unstable locus does not contain divisors. The latter is expressed as ${\rm cod}(\lambda)\geq 2$; such $\lambda$ and their GIT-classes are called movable. This motivates the introduction of the $\tilde G$-movable cone ${\rm Mov}^{\tilde G}(G/B)=\mc C_2(\iota)$. The GIT-classes which are movable chambers admit remarkably strong properties, exemplified by the following theorem. We refer the reader to \cite{Sepp-Tsan-2020} as well as to \cite{HuKeel} for the undefined notions, and note that the rest of this article does not refer to Mori theory. 

\begin{theorem}\label{Theo ST Mori} (cf. \cite[Th. II]{Sepp-Tsan-2020})

Suppose that $C\subset\mc C_2(\iota)={\rm Mov}^{\tilde G}(G/B)$ is a movable $\tilde G$-chamber and let $Y=Y_C$ be the GIT-quotient. Then $Y$ is a normal projective variety and a Mori dream space; the Picard group of $Y$ is naturally injected into ${\rm Pic}(G/B)=\Lambda$ as a sublattice of finite index, so that ${\rm Pic}(Y)_\RR\cong\Lambda_\RR$; the pseudoeffective cone of $Y$ is identified with the closure the $\tilde G$-ample cone $\ol{\rm Eff}(Y)\cong\ol{\mc C}_1(\iota)=\mc{LR}_0(\iota)$; the movable cone of $Y$ -- with the $\tilde G$-movable cone $\ol{\rm Mov}(Y)\cong\mc C_2(\iota)={\rm Mov}^{\tilde G}(G/B)$; the Mori chambers in $\ol{\rm Eff}(Y)$ are identified with the $\tilde G$-chambers in ${\mc C}_1(\iota)$; the Nef cone of $Y$ is identified with the closure of the chosen movable $\tilde G$-chamber ${\rm Nef}(Y)\cong \ol{C}$; the Cox ring of $Y$ is naturally embedded in the invariant ring ${\rm Cox}(G/B)^{\tilde G}=\oplus _\lambda V_\lambda^{\tilde G}$ and the latter is a finite extension of the former.
\end{theorem}

Movable chambers exist if and only if $\mc C_2$ has nonempty interior. Various numerical criteria for this are deduced in \cite{Sepp-Tsan-2020}. Clearly property (A-2) implies not mere existence, but every chamber being movable and $\Lambda_\RR^+$ being the closure of the union of chambers. Thus (A-2) is equivalent to the following.\\

\noindent{\bf Property (M):} $\ol{\rm Mov}^{\tilde G}(G/B)=\mc{LR}_0(\iota)=\Lambda_\RR^+$. Equivalently, any open GIT-class in $C\subset\Lambda_\RR^+$ is a movable chamber, Theorem \ref{Theo ST Mori} applies for the quotient $Y=Y_C$ and, in addition, $\ol{\rm Eff}(Y)=\ol{\rm Mov}(Y)\cong\Lambda_\RR^+$.\\

Theorem \ref{Theo tildeDelta ellK A} yields, as a special case, a sufficient condition for the above property to hold.

\begin{coro}\label{Coro all Move}
Let $\tilde G, G$ be connected reductive groups, $G$ semisimple. Suppose that at least one of the following conditions is satisfied:
\begin{enumerate}
\item[{\rm (i)}] $\ell_\Delta-\#\tilde\Delta> 1$;
\item[{\rm (ii)}] $\tilde w_0=-1$ and $\ell_\Delta^{\rm sd}-\#\tilde\Delta> 1$.
\end{enumerate}
Then $\ol{\rm Mov}^{\tilde G}(G/B)=\mc{LR}_0(\iota)=\Lambda_\RR^{+}$ and property (M) holds for any embedding $\iota:\tilde G\subset G$.
\end{coro}

\begin{rem}
In the presence of property (M), the $\tilde G$-chambers cover an open dense subset of $\Lambda^+_\RR$, and they are all movable. Thus it is much easier to obtain a weight $\lambda$ whose GIT-class is a movable $\tilde G$-chamber, and hence to obtain a GIT-quotient with the global properties described above. Indeed, in this situation, every $\tilde T$-chamber defines a movable $G$-chamber. In turn, the $\tilde T$-chambers are obtained as the connected components of the complement in $\Lambda_\RR^+$ of a union of hyperplanes determined combinatorially in \cite{Sepp-Tsan-2020}.
\end{rem}

\subsection{$SL_2$-subgroups}\label{Sect SL2 subgr}

In the following theorem we derive an exact criterion for the properties (A) and (M) for $SL_2$ subgroups of semisimple groups. Let us note that the initial general steps towards the description of $\mc{LR}_0$ for $SL_2$-subgroups have been outlined in \cite{Beren-Sjam}. Our methods allow to carry the calculations through and classify the exceptions to property (A). Furthermore, we also have access to property (A-$k$) and (M) in particular.

\begin{theorem}\label{Theo SL2}
Let $\iota:\tilde G\subset G$ be an $SL_2$-subgroup of a semisimple group $G$. Then the following hold:
\begin{enumerate}
\item[{\rm (i)}] $G$ has no simple factors of type ${\bf A}_1$ to which $\tilde G$ projects nontrivially, if and only if $\iota$ has property (A):
$$
\mc{LR}_0(\iota)=\Lambda_\RR^+ \;.
$$
\item[{\rm (ii)}] $G$ has no simple factors of type ${\bf A}_1,{\bf A}_2,{\bf B}_2$ to which $\tilde G$ projects nontrivially, if and only if $\iota$ has property (M):
$$
\ol{\rm Mov}^{\tilde G}(G/B)=  \mc{LR}_0(\iota)=\Lambda_\RR^+ \;.
$$
\item[{\rm (iii)}] Let $\mk g=\mk g_0\oplus\mk g_1\oplus...\oplus\mk g_k$ be the decomposition of $\mk g$, where $\mk g_1,...,\mk g_k$ are the simple ideals of rank 1 to which $\tilde{\mk g}$ projects nontrivially, and $\mk g_0$ is the ideal complementary to their sum. Let $h\in\tilde{\mk g}$ be an integral semisimple element and let $\Lambda^+$ be chosen so that $h$ is dominant. Any dominant weight $\lambda\in\Lambda^+$ can be written as $\lambda=\lambda_0+\lambda_1+...\lambda_k$, with $\lambda_j\in\Lambda^+_j$ dominant for $\mk g_j$. Then the cone $\mc{LR}_0$ is described by the following inequalities:
$$
\mc{LR}_0(\iota) = \{\lambda\in\Lambda^+: \; \lambda_j(h)\leq (\lambda_0+\lambda_1+...+\hat\lambda_j+...+\lambda_k)(h) \;,\; j=1,...,k \} \;,
$$
where $\hat{\cdot}$ indicates a skipped summand, as customary.
\end{enumerate}
\end{theorem}

\begin{proof}
Since for $\tilde{\mk g}\cong\mk{sl}_2$ we have $\tilde w_0=-1$, we can apply Theorem \ref{Theo tildeDelta ellK A},(ii), and we are brought to estimate the number $\ell_\Delta^{\rm sd}-\#\tilde\Delta^+=\ell_\Delta^{\rm sd}-1$. We use Corollary \ref{Coro ellDeltaDual}, reducing the value of $\ell^{\rm sd}_\Delta$ to the values for the simple components of $\Delta$, and Theorem \ref{Theo ellDelta Classification}, providing the values for the simple types.

Part (i) follows from the fact that for simple $\Delta$ we have $\ell_\Delta^{\rm sd}\geq 1$ with equality if and only if $\Delta$ is of type ${\bf A}_1$.

Part (ii) follows from the fact that, for $\Delta$ of rank at least 2, we have $\ell^{\rm sd}_\Delta\geq 2$, with equality exactly for $\Delta$ of type ${\bf A}_2$ or ${\bf B}_2$. Thus for all other cases we can apply Corollary \ref{Coro all Move} and obtain property (M). The fact that property (M) does not hold for $SL_2$-subgroups of $G$, when $G$ is of type ${\bf A}_2$ and ${\bf B}_2$, can be checked directly without difficulty.

For part (iii), first note that the proposed inequalities hold for $\mc{LR}_0(\iota)$. Indeed, for each $j\in\{1,...,k\}$, the cone $\mc{LR}_0(\iota)$ can be interpreted as the eigencone $\mc{LR}(\iota_j)$, where $\iota_j:\tilde{\mk g}\to \mk g_j^\perp$ is the projection to the complementary ideal to $\mk g_j$ in $\mk g$. Now the proposed inequality indexed by $j$ is an obvious necessary condition for the for $\lambda_j(h)$ to appear as a (highest) weight of a $\iota_j(\tilde{\mk g})$-submodule of the $\mk g_j^\perp$-module with highest weight $\lambda_{\vert\mk g_j^\perp\cap \mk t}$. To see that there are no other inequalities, we use the known description of $\mc{LR}_0(\iota)$, cited here in Theorem \ref{Theo SeppTsan Ck}. The inequalities describing $\mc{LR}_0(\iota)$ in the case of $SL_2$-subgroups all have the form $\langle \lambda,wh\rangle\geq 0$ for suitable elements of $l(w)=1$, i.e., simple reflections. The simple reflections corresponding to $\mk g_1,...\mk g_k$ yield the proposed inequalities. For all other simple reflections, the scalar product is always nonnegative as implied by part (i).
\end{proof}

\subsection{Subgroups of classical groups}\label{Sect Classical prop A}

Any classical group $G$ comes equipped with its natural representation $V$, and subgroups $\tilde G\subset G$ can be studied via their action on $V$. Below we apply this method to derive a criterion for a subgroup of a classical group to possess property (A-$k$).

\begin{theorem}\label{Theo Classical}
Let $G$ be a classical group with simple Lie algebra and $V\cong \CC^m$ be its natural representation. Let $\iota:\tilde G\subset G$ be a connected reductive subgroup. In case $\mk g\cong \mk{sl}_m$ assume in addition that $V\cong V^*$ as a $\tilde G$-module. Denote $k:=\lfloor (m-1)/2\rfloor - \#\tilde\Delta^+$, and $k=3-\#\tilde\Delta^+$ for the special case of $\mk g\cong\mk{so}_5$.

If $0<k$, then $k(\iota)\geq k$ and property (A-$k$) holds for $\iota$, so that
$$
\mc{LR}_0(\iota)=\ol{\mc C_k(\iota)}=\Lambda_\RR^+ \;.
$$
In particular, if $\dim V\geq\dim \tilde{\mk g}$, then ${\rm rank}(\tilde{\mk g})\geq 3$ implies property (A) and ${\rm rank}(\tilde{\mk g})\geq 5$ implies property (M). 
\end{theorem}

\begin{proof}
In case $G\cong SL_m$, the condition $V\cong V^*$ as $\tilde G$-module implies $\tilde\Gamma\subset\Gamma_{\rm sd}$, and puts us in a position to apply part (ii) of Theorem \ref{Theo tildeDelta ellK A}. We observe, from Theorem \ref{Theo ellDelta Classification}, that the values of $\ell_\Delta$ (resp. $\ell_\Delta^{\rm sd}$ for $SL_m$) for the classical groups satisfy $\ell_\Delta\geq\lfloor (m-1)/2\rfloor$. Therefore, $\ell_\Delta\geq \lfloor (m-1)/2\rfloor > \#\tilde\Delta^+$ (resp. $\ell_\Delta^{\rm sd}\geq \lfloor (m-1)/2\rfloor > \#\tilde\Delta^+$ for $SL_m$).

The last statement follows immediately, since $\dim\tilde{\mk g}=\dim\tilde{\mk t}+2\#\tilde\Delta^+$.
\end{proof}

\begin{example}
Let $\tilde{\mk g}$ be a simple Lie algebra of rank $n\geq 2$, and $V$ be a $\tilde{\mk g}$-module such that $V\cong V^*$ such that $\dim V<\dim\tilde{\mk g}$. Then, according to the isomorphism type of $\tilde{\mk g}$, $V$ is one of the following:
\begin{enumerate}
\item[${\bf A}_n$:] Irreducible modules: none in general; special cases: $V_{\varpi_{\frac{n+1}{2}}}$ for $n=3,5$.

Reducible modules: $(V_{\varpi_1}\oplus V_{\varpi_n})^{\oplus j}$ for $1\leq j<\frac{n+1}{2}$, $V_{\varpi_2}\oplus V_{\varpi_{n-1}}$;

special cases: $V_{\varpi_1}\oplus V_{\varpi_{\frac{n+1}{2}}}\oplus V_{\varpi_n}$ for $n=3,5$.
\item[${\bf B}_n$:] Irreducible modules: $V_{\varpi_1}$; special cases: $V_{\varpi_n} \;for\; n\leq 6$.

Reducible modules: $V_{\varpi_1}^{\oplus j} \;for\; 2\leq j<n$;

special cases: $V_{\varpi_n}^{\oplus 2}$ for $n\leq 4$; $V_{\varpi_1}\oplus V_{\varpi_n}$ for $n\leq 6$; $(V_{\varpi_1}\oplus V_{\varpi_n})^{\oplus2}$ for $n=4,5$.
\item[${\bf C}_n$:] Irreducible modules: $V_{\varpi_1}, V_{\varpi_2}$; special cases: $V_{\varpi_3}$ for $n\leq 3$.

Reducible modules: $V_{\varpi_1}^{\oplus j} \;for\; 2\leq j<n+\frac12$.

\item[${\bf D}_n$:] $(n\geq4)$. Irreducible modules: $V_{\varpi_1}$; special cases: $V_{\varpi_{n-1}},V_{\varpi_n}$ for $n= 4,6$.

Reducible modules: $V_{\varpi_1}^{\oplus j}$ for $2\leq j<n-\frac12$;

special cases: $V_{\varpi_1} \oplus V_{\varpi_{n-1}} , V_{\varpi_1} \oplus V_{\varpi_n},V_{\varpi_{n-1}}^{\oplus2} , V_{\varpi_n}^{\oplus2}$ for $n=4,6$,

$V_{\varpi_{n-1}} \oplus V_{\varpi_n}$ for $n\leq6$, $V\varpi_1\oplus V_{\varpi_{n-1}} \oplus V_{\varpi_n}$ for $n\leq5$.
\item[${\bf E}_6$:] Irreducible modules: none. Reducible modules: $V_{\varpi_1}\oplus V_{\varpi_6}$.
\item[${\bf E}_7$:] Irreducible modules: $V_{\varpi_1}\cong \CC^{56}$. Reducible modules: $V_{\varpi_1}^{\oplus2}$.
\item[${\bf E}_8$:] There are no such modules.
\item[${\bf F}_4$:] Irreducible modules: $V_{\varpi_1}\cong \CC^{26}$. Reducible modules: $V_{\varpi_1}^{\oplus2}$.
\item[${\bf G}_2$:] Irreducible modules: $V_{\varpi_1}$. Reducible modules: none.
\end{enumerate}
\end{example}

\subsection{$E_8$-subgroups}\label{Sect E8 Prop M}

It turns out that the simple group of type ${\bf E}_8$ possesses a special feature, which characterizes it among simple groups: property (M) holds for all its embeddings as a proper subgroup of a simple group. More generally, we have the following statement.

\begin{theorem}\label{Theo E8 all movable}
Suppose $\tilde G$ is of type ${\bf E}_8$ and let $\iota:\tilde G\subset G$ be an arbitrary embedding. Then $G$ has no simple factors to which $\tilde G$ projects surjectively if and only if property (M) holds and
$$
\ol{\rm Mov}^{\tilde G}(X)= \mc{LR}_0(\iota) = \Lambda_\RR^+ \;.
$$
In particular, property (M) holds for any embedding $\tilde G\subsetneq G$ with $G$ infinitesimally simple.
\end{theorem}

\begin{proof}
Since $\tilde G$ is of type ${\bf E}_8$ we have $\tilde w_0=-1$, and we can apply Theorem \ref{Theo tildeDelta ellK A},(ii) and more specifically Corollary \ref{Coro all Move},(ii) referring to the property (M). Thus it suffices to show that $\ell^{\rm sd}_\Delta>1+\#\tilde\Delta^+$ holds for any proper embedding $\iota:\tilde G\subset G$ into a simple group $G$. It is suitable to conduct the argument on the level of Lie algebras.

There are no embeddings of $\mk e_8$ into the other exceptional simple Lie algebras, so $\mk g$ must be classical. Any embedding into a classical Lie algebra comes together with the natural representation of $\mk g$, say $V$, the nontrivial $G$-module of smallest possible dimension. This dimension must exceed the smallest dimension of a nontrivial $\mk e_8$-module, which is obtained for the adjoint representation. Thus
$$
\dim V=\min\{\dim V_{\lambda}:\lambda\in\Lambda^+\setminus\{0\}\}\geq \min\{\dim \tilde V_{\tilde\lambda}:\tilde\lambda\in\tilde\Lambda^+\setminus\{0\}\}=\dim \mk e_8=8+2\#\tilde\Delta^+ \;.
$$
According to Theorem \ref{Theo ellDelta Classification} for all classical simple $\mk g$, the dimension of the natural representation and the number $\ell_\Delta^{\rm sd}$ satisfy $\ell_\Delta^{\rm sd}\geq \frac12\dim V - 1$. Hence
$$
\ell_\Delta^{\rm sd}\geq \frac12\dim V - 1 \geq 3+\#\tilde\Delta^+ > 1+\#\tilde\Delta^+ \;.
$$
This proves the statement.
\end{proof}

\subsection{The symplectic interpretation: projections of coadjoint orbits of compact groups}\label{Sect Compact}

In this paragraph, we recall the interpretation of $\mc{LR}_0(\iota)$ and $\mc{LR}(\iota)$ in terms of momentum maps, in a framework due to Heckman, \cite{Heckman82}, and translate our observations in this terminology. Connected complex reductive Lie groups are characterized as complexifications of compact connected Lie groups, and we have $G=K^\CC$, where $K$ is any maximal compact subgroup of $G$. The finite dimensional representation theories of $G$ and $K$ are well known to be equivalent. Any embedding of compact Lie groups complexifies to an embedding of complex reductive Lie groups, and conversely, any embedding of complex reductive Lie groups induces an embedding of compact groups, obtained by extending a maximal compact group of the subgroup to one of the ambient group. Thus our problems can be formulated entirely in terms of compact groups.

Let $\iota:\tilde G\subset G$ be an embedding connected reductive complex subgroup. Let $\tilde K$ and $K$ be respective maximal compact subgroups such that $\tilde K=\tilde G\cap K$. Let $\tilde{\mk g}=\tilde{\mk k}\oplus i\tilde{\mk k}\subset \mk k\oplus i\mk k=\mk g$ be the corresponding real form decompositions of the Lie algebras. Let $\tilde T\subset\tilde K$ and $T\subset K$ be maximal tori such that $\tilde T=\tilde K\cap T$. Let $\tilde H=\tilde T^\CC\subset T^\CC=H$ be the corresponding Cartan subgroups of complex groups. Let $\langle,\rangle$ be a $K$-invariant nondegenerate bilinear form on $\mk k$ extending the Killing form on the semisimple part; we use the same notation for the induced scalar product on $i\mk k$ and $i\mk k^*$. We have a natural inclusion of the weight lattice $\Lambda\subset i\mk t^*\subset i\mk k^*$ and $\Lambda_\RR=i\mk t^*$. Let $\Lambda^+$ be a Weyl chamber and $B\subset G$ be a Borel subgroup. 

Every (imaginary) coadjoint $K$-orbit in $i\mk k^*$ intersects $\Lambda_\RR^+$ at a single point and $\Lambda_\RR$ at a Weyl group orbit. We denote by $K\lambda\subset i\mk k^*$ the orbit of $\lambda\in\Lambda_\RR^+$ and by $K_\lambda=Z_K(\lambda)$ the isotropy group. The K\"ahler form of Kostant-Kirillov-Souriau makes $K\lambda=K/K_\lambda$ into a complex $G$-manifold, isomorphic as such to the flag variety $G/P_\lambda$. For integral $\lambda$ the K\"ahler structure is Hodge and defines the embedding corresponding to the line bundle $\mc L_\lambda$:
$$
K\lambda\cong K/K_\lambda \cong K[v_\lambda]=G[v_\lambda]=X_\lambda\subset \PP(V_\lambda) \;.
$$

Let $\mu:i\mk k^*\to i\tilde{\mk k}^*$ be the projection induced by $d\iota^*$. Then, for any fixed $\lambda\in\Lambda_\RR^+$, the restriction 
$$
\mu_{\vert_{K\lambda}} : K\lambda \to i\tilde{\mk k}^*
$$
is a $\tilde K$-equivariant momentum map. Let $\mc M(\lambda)=\mc M_{\tilde K}(K\lambda)=K\lambda\cap \mu^{-1}(0)$ be the $0$-fiber, also called the Kempf-Ness set. The quotient space $\mc M(\lambda)/\tilde K$ is the symplectic reduction of $K\lambda$ by the $\tilde K$-action, which is a K\"ahler space. For integral $\lambda\in\Lambda^+$, the symplectic reduction of $K\lambda$ is isomorphic to the projective variety obtained as the complex GIT-quotient $Y_\lambda=(X_\lambda)_{\tilde G}^{ss}//\tilde G$. The map between the two spaces is naturally defined by the orbit structures, as every semistable $\tilde G$-orbit intersects $\mc M(\lambda)$ along a single $\tilde K$-orbit (cf. \cite{Kirwan},\cite{Ness-StratNullcone}) and so
$$
(X_\lambda)_{\tilde G}^{ss} = \{x\in X_\lambda: \ol{\tilde Gx}\cap\mc M(\lambda) \}\;.
$$
Thus, for nonzero integral $\lambda\in\Lambda^+\setminus\{0\}$, the existence of nonconstant elements in the invariant ring $\mc R(\lambda)^{\tilde G}$ is equivalent to $\mc M(\lambda)\ne\emptyset$. The Kempf-Ness set for coadjoint orbits can be expressed as $\mc M(\lambda)=K\lambda\cap i\tilde{\mk k}^\perp$. The latter is well defined for any $\lambda\in\Lambda_\RR^+$ and the $0$-eigencone can be written as
$$
\mc{LR}_0(\iota)=\{\lambda\in\Lambda_\RR^+:\mc M(\lambda)\ne \emptyset\} = \Lambda_\RR^+\cap Ki{\mk k}^\perp \;.
$$
Thus our property (A) demanding $\mc{LR}_0(\iota)=\Lambda_\RR^+$ can be formulated in this context as the property (A'), for embeddings $\iota:\tilde K\subset K$ of compact connected Lie groups, given below. For simplicity of expression we refer to the adjoint representation, which is isomorphic over $K$ to the imaginary coadjoint.\\

\noindent{\bf Property (A'):} The orthogonal projection of every adjoint $K$-orbit in $\mk k$ to the subalgebra $\tilde{\mk k}$ contains $0$. Equivalently, $\mc M_{\tilde K}(Kx)=Kx\cap \tilde{\mk k}^\perp\ne\emptyset$ for every $x\in\mk k$. Equivalently, $K\tilde{\mk k}^\perp =\mk k$.\\



\section{Proof of the criterion}\label{Sect Proof of Main}

In this section we supply a proof of Theorem \ref{Theo tildeDelta ellK A}. In fact, the particular case of the statement concerning $\mc{LR}_0(\iota)$ can be derived in a fairly straightforward way from the known form of the inequalities describing the cone, as given in \cite{Beren-Sjam}, \cite{Ress-Eigencone}, \cite{Sepp-Tsan-2020}. Since we are also interested in the cones $\mc C_k(\iota)$, we adhere to the formalism of \cite{Sepp-Tsan-2020}. The key step is based on explicit formulae for $X^{us}_{\tilde G}(\lambda)$ and its codimension in $X$, denoted by ${\rm cod}_{\tilde G}(\lambda)$. These formulae, given here in Theorem \ref{Theo SeppTsan Xuslambda cod}, are derived from the Hilbert-Mumford criterion and the Kirwan-Ness stratification theorem. The general theory relevant for this setting is provided for instance in \cite{Ness-StratNullcone}, and is also reviewed for the case at hand in \cite{Ress-Eigencone}, \cite{Sepp-Tsan-2020}.\\

We need to introduce some notions and notation concerning the structure of $G$ and its flag variety $X=G/B$. Let $n={\rm rank}(G)$ be the rank of the root system $\Delta$ of $G$. Let $\Gamma=\Lambda^\vee\in\mk h$ denote the lattice of one-parameter subgroups (OPS) and let $\langle,\rangle$ denote the bilinear pairing between $\Lambda$ and $\Gamma$. Let $\Gamma^+$ be the Weyl chamber in $\Gamma$ dual to $\Lambda^+$. For $h\in\Gamma$, let $P_h\subset G$ denote the parabolic subgroup defined by $h$, that is, the subgroup whose Lie algebra $\mk p_h$ is spanned by $\mk h$ and the root spaces $\mk g_\alpha$ with $\langle\alpha,h\rangle\geq 0$.

The $H$-fixed points in $X$ are parametrized by the Weyl group, $X^H=\{x_w=\dot{w}B: w\in W\}$, where $\dot{w}\in N_G(H)$ is an arbitrary representative of $w\in W$. The $B$-orbits through the $H$-fixed points exhaust $X$ and form the Schubert cell decomposition $X=\sqcup_w Bx_w$. The length of an element $w\in W$ is defined as the dimension of $Bx_w$ and also characterized as follows:
$$
l(w)=\dim Bx_w = \#(\Delta\cap w^{-1}\Delta^-) = \min\{l\in\ZZ_{\geq 0}:\exists j_1,...,j_l\in\{1,...,n\},w=r_{j_1}...r_{j_n}\},
$$
where $r_1,...,r_n\in W$ are the simple reflections with respect to the chosen Weyl chamber generating $W$. The codimension of the Schubert cell $Bx_w$ is equal to the dimension of the orbit of the opposite Borel subgroup $B^-$ through $x_w$, and can be expressed as
$$
{\rm codim}_X Bx_w = \dim B^-x_w= l(ww_0)=l(w_0w) = l(w_0)-l(w)=\#\Delta^+-l(w)\;,
$$
where $w_0$ is the longest element of $W$, having $l(w_0)=\#\Delta^+$ and sending $\Gamma^+$ to $\Gamma^-=-\Gamma^+$.

More generally, for $h\in \Gamma\setminus \{0\}$, we let $\sigma_h\in W$ be an element of minimal length such that $h\in\sigma_h\Gamma^+$, or equivalently, $B^{\sigma_h}=\sigma_h B\sigma_h^{-1}\subset P_h$. We define 
the $h$-length of $w\in W$ as the dimension of the $B^{\sigma_h}$-orbit through $x_{w\sigma_h}$:
$$
l_{h}(w)=\dim B^{\sigma_h} x_{w\sigma_h}=l(\sigma_h^{-1}w\sigma_h) \;.
$$
For a pair $(h,w)\in\Gamma\times W$, we consider the parabolic orbit $P_hx_{w\sigma_h}=\sqcup_{w'\in W_h w} B^{\sigma_h} x_{w\sigma_h}$. The dimension and codimension of this orbit are denoted by
\begin{align}\label{For cod h w}
p_h(w)=\dim P_hx_{w\sigma_h} & = \max\{l_h(w'): w'\in W_h w\} \;,\\
{\rm cod}_h(w)={\rm codim}_X P_hx_{w\sigma_h} & = \#\Delta^+-p_h(w)=\min\{l_{-h}(w'):w'\in W_hw\} \;.
\end{align}

For a pair $(\lambda,h)\in \Lambda^+\times(\Gamma\setminus\{0\})$, $\sigma_h\lambda$ and $h$ are dominant with respect to the same Borel subgroup. We consider the following partition of the Weyl group
\begin{align*}
& W=W^+(\sigma_h\lambda,h)\sqcup W^0(\sigma_h\lambda,h)\sqcup W^-(\sigma_h\lambda,h) \;,\\
& W^+(\sigma_h\lambda,h)= \{w\in W: \langle w\sigma_h\lambda,h\rangle>0\} \;,\\
& W^0(\sigma_h\lambda,h)= \{w\in W: \langle w\sigma_h\lambda,h\rangle=0\} \;,\\
& W^-(\sigma_h\lambda,h)= \{w\in W: \langle w\sigma_h\lambda,h\rangle<0\} \;.
\end{align*}

\begin{definition}
For $h\in\Gamma\setminus\{0\}$, we define the function
$$
\ell_h^-: \Lambda^+ \to \ZZ_{\geq 0} , \ell_h^-(\lambda)=\min\{l\in\ZZ_{\geq 0}:l=l_h(w), w\in W^-(\sigma_h\lambda,h)\} \;,
$$
and denote its minimum on nonzero dominant weights by
$$
\ell^h = \min\{\ell_h^-(\Lambda^+\setminus\{0\})\} \;.
$$
\end{definition}

The invariant $\ell_\Delta$ defined in (\ref{For Def ellK}) satisfies:
\begin{gather*}
\ell_\Delta = \min\{ \ell^h: h\in\Gamma\setminus\{0\}\} =\min\{ \ell^h: h\in\Gamma^+\setminus\{0\}\}  \;.
\end{gather*}
As for the case of weights, the dual element to $h\in\Gamma^+$ is the element $h^*=-w_0h\in\Gamma^+$ and $\Gamma_{\rm sd}^+=\{h\in\Gamma^+:h=h^*\}$ denotes the set of self-dual dominant elements. We also denote by $\Gamma_{\rm sd}=W\Gamma^+$ the set consisting of the elements self-dual with respect to any Weyl chamber to which they belong. We have
$$
\ell^{\rm sd}_\Delta = \min\{ \ell^h: h\in\Gamma_{\rm sd}\setminus\{0\}\} =\min\{ \ell^h: h\in\Gamma_{\rm sd}^+\setminus\{0\}\} \;.
$$

Given an embedding $\iota:\tilde G\subset G$, we choose Cartan and Borel subgroups, so that $\tilde H=\tilde G\cap H$, $\tilde B=\tilde G\cap B$, and furthermore $\tilde\Gamma_\RR^+\cap\Gamma_\RR^+$ contains an open subset of $\tilde\Gamma_\RR$; such Weyl chambers are called compatible. Note that $\tilde\Gamma\subset\Gamma$, but a Weyl chamber of $\tilde G$ is not necessarily contained in a Weyl chamber of $G$. For $h\in\tilde\Gamma^+\setminus\{0\}$, the element $\sigma_h$ is to be chosen with the supplementary condition that $\sigma_h\Gamma^+$ and $\tilde\Gamma^+$ are compatible Weyl chambers and furthermore $\sigma_h$ is of minimal length in its coset by the Weyl group of $Z_G(\tilde T)$. (See \cite{Beren-Sjam}, \cite{Sepp-Tsan-2020} for more detail on the elements $\sigma_h$ and the related notion of cubicles. The requirements for $\sigma_h$ ensure that the notions of dual and self-dual elements in $\tilde\Gamma$ and $\Gamma$ are compatible). 

\begin{definition}\label{Def RessElt}
Let $\mk P(\iota)=\{P_h: h\in\tilde\Gamma^+\setminus\{0\}\}$ be the set of parabolic subgroups of $G$ defined by nonzero OPS of $\tilde H$. Let $\{P_1,...,P_m\}=\mk P_{\rm max}(\iota)$ be the set of maximal elements of $\mk P(\iota)$ with respect to inclusion. It is shown in \cite{Sepp-Tsan-2020} that there exist unique indivisible (i.e. $\frac{1}{k}h\notin \Gamma$ for $k>1$) elements $h_1,...,h_m\in\tilde\Gamma^+$, such that $P_j=P_{h_j}$ for $j=1,...,m$. We denote $\mk R=\{h_1,...,h_m\}$, these element are called the Ressayre one-parameter subgroups for the embedding $\iota$ (see \cite{Ress-Eigencone} for the original definition). We denote by $\sigma_j=\sigma_{h_j}$ the corresponding Weyl group elements.
\end{definition}

The formula for the $\tilde G$-unstable locus in $X$ presented below involves varieties of the form $\tilde GP_{h}x_{w\sigma_h}$. Note that $\tilde GP_{h}x_{w\sigma_h}$ is the target of a natural $\tilde G$-equivariant surjective map from the homogeneous vector bundle $\tilde G\times_{\tilde P_h} P_h x_{w\sigma_h}$ on $\tilde G/\tilde P_h$:
$$
\tilde G\times_{\tilde P_h} P_h x_{w\sigma_h} \to \tilde GP_h x_{w\sigma_h}\;,\; [g,px_{w\sigma_h}]=gpx_{w\sigma_h}\;.
$$
Recalling that $\#\tilde\Delta_h^+=\dim \tilde G/\tilde P_h$, we obtain
$$
{\rm cod}_{\tilde G}(h,w)={\rm codim}_X \tilde GP_hx_{w\sigma_h}\geq {\rm cod}_h(w)-\#\tilde\Delta_h^+\;. 
$$

\begin{definition}
A pair $(h,w)\in\tilde\Gamma^+\times W$ is called fit, if $\dim\tilde GP_hx_{w\sigma_h}=\dim\tilde G/\tilde P_h+\dim P_hx_{w\sigma_h}$. With the above notation this condition is expressed as
$$
{\rm cod}_{\tilde G}(h,w)={\rm cod}_h(w)-\#\tilde\Delta_h^+\;.
$$
We denote by $\mk{F}$ the set of fit pairs with the supplementary condition that $h$ is indivisible and $w$ is of maximal $h$-length in $W_hw$. Thus for $(h,w)\in\mk{F}$, we have ${\rm cod}_{\tilde G}(h,w)=\#\Delta^+-l_h(w)-\#\tilde\Delta_h^+$. We denote by
$$
\mk{FR} = \mk{F}\cap (\mk R\times W)
$$
the set of fit pairs where the one-parameter subgroup is a Ressayre element as in Definition \ref{Def RessElt}.
\end{definition}

\begin{rem}
There is an obvious necessary condition for a pair $(h,w)\in\tilde\Gamma^+\times W$ to be fit, namely, $\dim P_hx_{w\sigma_h}+\dim\tilde G/\tilde P_h\leq \dim X$. For $SL_2$-subgroups, as well as for tori, this condition is also sufficient.
\end{rem}

For $\lambda\in\Lambda^+$ we denote
\begin{align*}
& (\tilde\Gamma^+\times W)_\lambda^+=\{(h,w)\in\tilde\Gamma^+\times W: \langle w\sigma_h\lambda,h \rangle > 0\}\;,\\
& (\mk R\times W)_\lambda^+=(\mk R\times W)\cap(\tilde\Gamma^+\times W)_\lambda^+\;, \\
& \mk{FR}^+_\lambda=\mk{FR}\cap(\tilde\Gamma^+\times W)_\lambda^+\;.
\end{align*}

\begin{theorem}\label{Theo SeppTsan Xuslambda cod} (cf. \cite[Theorem 3.5, Corollary. 4.6]{Sepp-Tsan-2020})

Let $\lambda\in\Lambda^{++}$. The $\tilde G$-unstable locus in $X=G/B$ with respect to $\mc L_\lambda$ is given by the following expressions: 
$$
X^{us}_{\tilde G}(\lambda) = \bigcup\limits_{(h,w)\in(\tilde\Gamma^+\times W)_\lambda^+} \tilde GP_h x_{w\sigma_h} = \bigcup\limits_{j=1}^m \bigcup\limits_{w\in W:\langle w\sigma_h\lambda,h_j\rangle>0} \tilde GP_j x_{w\sigma_h} \;.
$$
The codimension ${\rm cod}_{\tilde G}(\lambda)={\rm codim}_X X^{us}_{\tilde G}(\lambda)$ is given by:
\begin{align*}
{\rm cod}_{\tilde G}(\lambda) &= \min\{ {\rm cod}_{\tilde G}(h_j,w): (h_j,w)\in\mk{FR}^+_\lambda  \} \\
 &= \min\{\#\Delta^+-\#\tilde\Delta_{h_j}-l_{h_j}(w):(h_j,w)\in\mk{FR}^+_\lambda\}\;.
\end{align*}
In particular, the inequality 
$$
{\rm cod}_{\tilde G}(\lambda)\geq \min\{\#\Delta^+-\#\tilde\Delta_{h}-l_{h}(w):(h_j,w)\in(\tilde\Gamma^+\times W)^+_\lambda\}
$$
holds.
\end{theorem}

\begin{theorem}\label{Theo SeppTsan Ck} (cf. \cite[Theorem 5.8]{Sepp-Tsan-2020})

For $k\in\NN$, the cone $\mc C_k(\iota)\subset C^{\tilde G}(G/B)\subset \Lambda_\RR^{++}$ associated to the embedding $\iota$ is described as
$$
\mc C_k(\iota) = \{\lambda\in\Lambda_\RR^{++}:\; \langle \lambda , \sigma_h^{-1}w^{-1} h \rangle\leq 0\;,\forall(h,w)\in\mk{FR},l_{h}(w)=\#\Delta^+-\#\tilde\Delta_{h}^+-k+1 \}\;.
$$
\end{theorem}

In view of these results, the proof of Theorem \ref{Theo tildeDelta ellK A} is a matter of several elementary steps which we now present.

\begin{proof}[Proof of Theorem \ref{Theo tildeDelta ellK A}]

First remark that $k(\iota)\geq 0$ by definition, so the statement of part (i) of the theorem is redundant in case $\ell_\Delta\leq \#\tilde\Delta^+$. Analogously, if $\ell^{\rm sd}_\Delta\leq \#\tilde\Delta^+$, then part (ii) is redundant.

We construct an auxiliary family of cones related to the cones $\mc C_k(\iota)$. For any subset $\Xi\subset \Gamma\setminus\{0\}$, we consider the following cones in the Weyl chamber $\Lambda_\RR^+$:
$$
\mc E_l(\Xi)=\{\lambda\in\Lambda_\RR^+: l\leq\ell^-_h(\lambda), h\in\Xi\} \;,\; l\in \NN\;.
$$
These cones form a descending sequence $\Lambda^+_\RR=\mc E_1(\Xi)\subset\mc E_2(\Xi)\supset...$ which eventually hits $\{0\}$ (this follows from \cite[Lemma 4.1]{Sepp-Tsan-2020}). Also note that $\mc E_l(\Xi_1)\subset\mc E_l(\Xi_2)$, whenever $\Xi_2\subset\Xi_1$. 

The numbers $\ell_\Delta$ and $\ell^{\rm sd}_\Delta$ are characterized by the property that
$$
\mc E_{\ell_\Delta}(\Gamma\setminus\{0\})=\Lambda_\RR^+ \quad,\quad \mc E_{\ell^{\rm sd}_\Delta}(\Gamma_{\rm sd}\setminus\{0\})=\Lambda_\RR^+
$$
and $\ell_\Delta$ (resp. $\ell^{\rm sd}_\Delta$) is the maximal $l$ for which $\mc E_{l}(\Gamma\setminus\{0\})$ (resp. $\mc E_{l}(\Gamma_{\rm sd}\setminus\{0\})$) is equal to the entire Weyl chamber.

\begin{lemma}\label{Lemma ElCk}
For $l\geq\#\tilde\Delta^+$, if $\Lambda^{++}\cap \mc E_{l}(-\mk R)\ne \emptyset$, then
$$
\mc E_{l}(-\mk R)\subset \ol{\mc C}_{l-\#\tilde\Delta^+}(\iota)\;.
$$
\end{lemma}

\begin{proof}
Assume that $l\geq\#\tilde\Delta^+$ and $\mc E_{l}(-\mk R)\ne\{0\}$. Let $\lambda\in \Lambda^{++}\cap \mc E_{l}(-\mk R)\setminus\{0\}$. For any $(h,w)\in\mk{FR}_\lambda^+$ we have $\langle w\sigma_h\lambda,h\lambda\rangle>0$ and hence, by (\ref{For cod h w}), ${\rm cod}_h(w)\geq \ell^-_{-h}(\lambda)$. We obtain
$$
{\rm cod}_G(h,w)\geq {\rm cod}_h(w)-\#\tilde\Delta_h^+\geq \ell_{-h}^-(\lambda)-\#\tilde\Delta_h^+ \geq l-\#\tilde\Delta^+\;.
$$
Therefore, according to Theorem \ref{Theo SeppTsan Ck}, we have $\lambda\in\ol{\mk C}_{l - \#\tilde\Delta^+}(\iota)$, which proves the lemma.
\end{proof}

We can now prove the statements of the theorem. If $\ell_\Delta>\#\tilde\Delta^+$ holds, then Lemma \ref{Lemma ElCk} implies
$$
\Lambda_\RR^+=\mc E_{\ell_\Delta}(\Gamma\setminus\{0\})\subset \mc E_{\ell_\Delta}(-\mk R)\subset 
\ol{\mc C}_{\ell_\Delta-\#\tilde\Delta^+}(\iota) \;.
$$
This proves part (i).

Suppose now $\tilde{w}_0=-1$ holds for the group $\tilde G$. This means that all elements of $\tilde\Gamma^+$ are self dual, and hence $\tilde\Gamma^+\cap\Gamma^+\subset\Gamma^+_{\rm sd}$ for any choice of compatible Weyl chambers. Hence $\tilde\Gamma\subset \Gamma_{\rm sd}$ and in particular $\pm\mk R\subset \Gamma_{\rm sd}$. Now, if the inequality $\ell^{\rm sd}_\Delta>\#\tilde\Delta^+$ holds, then Lemma \ref{Lemma ElCk} implies
$$
\Lambda_\RR^+=\mc E_{\ell^{\rm sd}_\Delta}(\Gamma_{\rm sd}\setminus\{0\})\subset \mc E_{\ell^{\rm sd}_\Delta}(-\mk R)\subset \ol{\mc C}_{\ell^{\rm sd}_\Delta-\#\tilde\Delta^+}(\iota) \;.
$$
This proves part (ii) and completes the proof of the theorem.
\end{proof} 

\section{Proof of Theorem \ref{Theo ellDelta Classification}: values of $\ell_G$ and $\ell_G^{\rm sd}$}\label{Sect Compute}

Let $G$ be a semisimple complex Lie group with Lie algebra $\mk g$ and root system $\Delta$. Here we study the numerical invariants $\ell_\Delta$ and $\ell_\Delta^{\rm sd}$ defined in formula (\ref{For Def ellK}) in the introduction. In particular, we compute their values for simple root systems, as claimed in Theorem \ref{Theo ellDelta Classification}.

This section is independent from the rest of the article and is concerned solely with combinatorics of Weyl groups, roots and weights. In order to make the exposition independent, we briefly reintroduce the necessary notation:

\begin{enumerate}
\item[$\bullet$] $\Delta=\Delta^+\sqcup\Delta^-$ root system of rank $n$, split into a positive and a negative part.
\item[$\bullet$] $\alpha_1,...,\alpha_n\in\Delta^+$ simple roots.
\item[$\bullet$] ${\rm hi}:\Delta^+\to\NN$, ${\rm hi}(\sum m_j\alpha_j)=\sum m_j$, height function.
\item[$\bullet$] $\Delta^\vee=\{\alpha^\vee=\frac{2}{\langle\alpha,\alpha\rangle}\alpha:\alpha\in\Delta\}$, dual root system, with height function $\check{\rm hi}$.
\item[$\bullet$] $\Lambda$ weight lattice.
\item[$\bullet$] $\langle,\rangle$ scalar product on $\Lambda_\RR$ induced by the Killing form.
\item[$\bullet$] $\Lambda^+\subset\Lambda_\RR^+$ integral and real Weyl chambers corresponding to $\Delta^+$. 
\item[$\bullet$] $\varpi_1,...,\varpi_n\in\Lambda^+$ fundamental weights.
\item[$\bullet$] $\rho=\varpi_1+...+\varpi_n=\frac12\sum\limits_{\alpha\in\Delta^+} \alpha$, the smallest strictly dominant weight.
\item[$\bullet$] $W$ Weyl group generated by the simple root reflections $r_1,...,r_n$.
\item[$\bullet$] Given $\lambda\in\Lambda$ and an expression $w=r_{j_1}...r_{j_l}\in W$, we call the sequence of weights $\lambda_i=r_{j_{l-i+1}}...r_{j_l}\lambda$, for $i=1,...,l$, the $w$-trajectory of $\lambda$. If an expression is not specified in the discussion, the term refers to the trajectory of arbitrary expression for $w$.
\item[$\bullet$] Length function $l:W\to \ZZ$, $l(w)=\#(\Delta^+\cap w^{-1}\Delta^-)=\min\{l\in\ZZ_{\geq 0}: \exists j_1,...,j_l\in\{1,...,n\}, w=r_{j_1}...r_{j_l}\}$.
\item[$\bullet$] $W(l)=\{w\in W: l(w)=l\}$.
\item[$\bullet$] $w_0\in W$ longest element. We have $l(w_0)=\#\Delta^+$ and $l(w_0w)=l(ww_0)=l(w_0)-l(w)$ for all $w\in W$.
\item[$\bullet$] $\Lambda^+_{\rm sd}=\{\lambda\in\Lambda^+:\lambda=\lambda^*\}$, where $\lambda^*=-w_0\lambda$. Similarly for $\Gamma^+_{\rm sd}$.
\item[$\bullet$] For a pair $(\lambda,h)\in \Lambda_\RR^+\times\Lambda_\RR^+$ we define a partition of the Weyl group by
\begin{align*}
& W=W^+(\lambda,h)\sqcup W^0(\lambda,h)\sqcup W^-(\lambda,h) \;,\\
& W^+(\lambda,h)= \{w\in W: \langle w\lambda,h\rangle>0\} \;,\\
& W^0(\lambda,h)= \{w\in W: \langle w\lambda,h\rangle=0\} \;,\\
& W^-(\lambda,h)= \{w\in W: \langle w\lambda,h\rangle<0\} \;.
\end{align*}
\end{enumerate}

\begin{definition}\label{Def ellh}
For $h\in\Gamma^+\setminus\{0\}$, we denote
\begin{align*}
& \ell_h^-: \Lambda^+ \to \ZZ_{\geq 0} \;,\; \ell_h^-(\lambda)=\min\{l\in\ZZ_{\geq 0}:l=l(w), w\in W^-(\lambda,h)\} \;,\\
&\ell^h = \min\{\ell_h^-(\Lambda^+\setminus\{0\})\} \;.
\end{align*}
We define two numerical invariants of the root system $\Delta$ by
\begin{align*}
& \ell_\Delta = \min\{ \ell^h: h\in\Lambda^+\setminus\{0\}\} \;,\\
& \ell_\Delta^{\rm sd} = \min\{ \ell^h: h\in\Lambda^+_{\rm sd} \setminus\{0\}\} \;.
\end{align*}
\end{definition}

\begin{rem}\label{Rem h in Lambda}
The element $h$ above is in fact to be interpreted, via the Killing form, as a Lie algebra element in the coweight lattice $\Gamma=\Lambda^\vee\subset\mk h$, in order to match our considerations in the rest of the article. Here we choose to identify $\Lambda_\RR\cong\Gamma_\RR$ via the Killing form, while keeping the notation $h$ to indicate the two roles. This is convenient because it enables better formulation of the symmetry properties of the above partitions of $W$. It is also harmless, since $\Lambda_\RR^+\cong\Gamma_\RR^+$ and moreover the partition of $W$ remains the same if $h$ or $\lambda$ is replaced by a positive multiple. Hence $\ell_h^-$ and $\ell^h$ are also unchanged by positive rescaling of $h$. Note that the fundamental coweight $\varpi_j^\vee$ is a positive multiple of $\varpi_j$; this is important, since the values $\ell^-_{\varpi_j}(\varpi_k)$ and $\ell^{\varpi_j}$ play a key role in what follows. The distinction manifests for some elements, such as the principal semisimple element $\rho^\vee=\varpi_1^\vee+...+\varpi_n^\vee$, which is not proportional to $\rho$ if $\Delta\ncong \Delta^\vee$. Nonetheless, $\ell_\Delta=\ell_{\Delta^\vee}$ holds, as shown in the corollary below.
\end{rem}

\begin{example}\label{Exa shortroot}
Let $h=\beta\in \Delta\cap\Lambda^+$ be a dominant root, and suppose, without loss of generality, that $\Delta$ is simple.
\begin{enumerate}
\item[{\rm (i)}] If $\Delta$ is simply laced, then $\ell^\beta={\rm hi}(\beta)$.
\item[{\rm (i)}] If $\Delta$ is has two root lengths, then
$$
\ell^\beta=\begin{cases} {\rm hi}(\beta) \;\;,\;if\; \beta\; is\; short\;,\\
\check{\rm hi}(\beta^\vee) \;\;,\;if\; \beta\; is\; long\;. \end{cases}
$$
\end{enumerate}
To show this, note first that, for any dominant root $\beta$, $\ell^\beta$ is the length of the shortest Weyl group element sending $\beta$ to a negative root. Every simple reflection $r_j$ sends exactly one positive root, $\alpha_j$, to a negative. If $\alpha$ is a positive root, short if there are two root lengths, then the reflection coefficient $n_{\alpha_j,\alpha}$ has four possible values $0,1,-1,2$. The value $2$ occurs exactly for $\alpha=\alpha_j$, in which case $r_j(\beta)=-\alpha_j$. It follows that a Weyl group element $w$ of length at least ${\rm hi}(\alpha)$ is necessary to attain $w(\alpha)\in \Delta^-$. The existence of such an element is a standard fact, and the image $w(\alpha)$ is necessarily the negative of a simple root, the one corresponding to the last simple reflection in an expression for $w$. By applying this argument to $\beta$ we obtain statement (i) and statement (ii) for short $\beta$. The statement for long $\beta$ follows, since $\beta^\vee$ is short, and $\ell^{\beta}=\ell^{\beta^\vee}$.
\end{example}

\begin{lemma}\label{Lemma lminush min bj}
For any $h\in\Lambda_\RR^+\setminus\{0\}$, the following hold:
\begin{enumerate}
\item[{\rm (i)}] $\forall \lambda\in\Lambda_\RR^+$, the map $W^\varepsilon(\lambda,h)\to W^\varepsilon(h,\lambda)$, $w\mapsto w^{-1}$ is a length preserving bijection, for $\varepsilon\in\{+,0,-\}$;
\item[{\rm (ii)}] $\forall \lambda\in\Lambda_\RR^+$, $\ell^-_h(\lambda)=\ell^-_\lambda(h)$;
\item[{\rm (iii)}] $\forall\lambda_1,\lambda_2\in\Lambda^+,\; \ell_h^-(\lambda_1+\lambda_2)\geq \min\{\ell_h^-(\lambda_1),\ell_h^-(\lambda_2)\}$;
\item[{\rm (iv)}] $\ell^h=\min\{\ell_h^-(\varpi_j):j=1,...,n\}$.
\end{enumerate}
\end{lemma}

\begin{proof}
Part (i) follows from the equalities $\langle w\lambda, h\rangle=\langle\lambda, w^{-1}h\rangle=\langle w^{-1}h,\lambda,\rangle$ and the fact that $(r_{j_1}...r_{j_l})^{-1}=r_{j_l}...r_{j_1}$. Part (ii) follows from part (i) and the definition of $\ell^-_h$. Part (iii) follows from the simple fact, if $\langle w\lambda_j,h\rangle>0$ holds for $j=1,2$, then $\langle w(\lambda_1+\lambda_2),h\rangle>0$. Part (iv) is a direct consequence of (iii) and the fact that the fundamental weights generate $\Lambda^+$.
\end{proof}

\begin{coro}\label{Coro ellDeltaDual}
The invariant $\ell_\Delta$ admits the following properties:
\begin{enumerate}
\item[{\rm (i)}] $\ell_\Delta=\min\{\ell^{\varpi_1},...,\ell^{\varpi_n}\}$;
\item[{\rm (ii)}] $\ell_\Delta=\min\{\ell^-_{\varpi_j}(\varpi_k)\;: \; j,k\in\{1,...,n\}\}$;
\item[{\rm (iii)}] $\ell_\Delta=\ell_{\Delta^\vee}$, where $\Delta^\vee$ is the dual root system;
\item[{\rm (iv)}] $\ell_\Delta=\min\{\ell_{\Delta_1},...,\ell_{\Delta_k}\}$, where $\Delta=\Delta_1\sqcup\dots\sqcup\Delta_k$ is the decomposition of $\Delta$ into simple components.
\end{enumerate}
The properties (iii) and (iv) hold for the invariant $\ell_\Delta^{\rm sd}$ as well. In addition, we have
$$
\ell_\Delta^{\rm sd}=\min\{\min\{\ell_h^-(\varpi_j):h\in\Lambda_{\rm sd}^+\setminus\{0\},j=1,...,n\}\}\;.
$$
\end{coro}

\begin{proof}
Parts (i) and (ii) follow from Lemma \ref{Lemma lminush min bj},(ii),(iv), as
\begin{align*}
\ell_\Delta&=\min\{\ell^h:h\in\Lambda^+\setminus\{0\}\}=\min\{\ell_h^-(\varpi_j):h\in\Lambda^+\setminus\{0\},j=1,...,n\}\\
 & =\min\{\ell^-_{\varpi_j}(h):h\in\Lambda^+\setminus\{0\},j=1,...,n\}=\min\{\ell^-_{\varpi_j}(\varpi_k):j,k=1,...,n\}\\
 & = \min\{\ell^{\varpi_j}:j=1,...,n\}\;.
\end{align*}
The fact that $\ell_\Delta$ is attained at a fundamental weight implies parts (iii) and (iv). Indeed, the fundamental weights $\varpi_j^\vee$ of the dual system are positively proportional to $\varpi_j$ (see Remark \ref{Rem h in Lambda}), whence $\ell^{\varpi_j}=\ell^{\varpi_j^\vee}$, and this implies part (iii). For part (iv) note that the fundamental weights of $\Delta$ are the fundamental weights of its simple components and the Weyl group is a direct product.

Analogously, the given expression for $\ell_\Delta^{\rm sd}$ follows from Lemma \ref{Lemma lminush min bj},(iv). For property (iii), note that $\Delta\ncong\Delta^\vee$ implies $w_0=-1$ and $\Lambda^+=\Lambda^+_{\rm sd}$. Thus, whenever the property $\ell_\Delta^{\rm sd}=\ell_{\Delta^\vee}^{\rm sd}$ is not redundant, it follows from the property $\ell_\Delta=\ell_{\Delta^\vee}$. Property (iv) follows from the given expression and the fact that $\Lambda_{\rm sd}^{+} = (\Lambda_1^{+})_{\rm sd}\times\dots\times (\Lambda_k^{+})_{\rm sd}$.
\end{proof}

With this general preparation, we are ready to formulate our theorems on the values of $\ell_\Delta$ and $\ell_\Delta^{
\rm sd}$. The proofs consist of case analysis of the simple types. We present general proofs for the classical series in corresponding subsections. The results for the exceptional groups are obtained in part with the help of a computer calculation, and we present tables with the necessary data.

\begin{theorem}\label{Theo Values ellDelta}
The values of the numerical invariant $\ell_{\Delta}$ for a simple root system $\Delta$ are as follows:

\begin{tabular}{|l|l|l|l|}
\hline
Type of $\Delta$ & $\ell_\Delta$ & Attained at $\ell_h^-(\lambda)$ & $w\in W(\ell_\Delta):\langle w\lambda,h\rangle<0$ \\
\hline
${\bf A}_n$ & $1$ & $\ell_{\varpi_1}^-(\varpi_1)$ & $r_1$ \\
\hline
${\bf B}_n$ & $n$ & $\ell_{\varpi_1}^-(\varpi_n)$ & $r_1\dots r_n$ \\
            &     & & \\
\hline
${\bf C}_n$ & $n$ & $\ell_{\varpi_1}^-(\varpi_n)$ & $r_1\dots r_n$ \\
            &     & & \\
\hline
${\bf D}_n$ & $n-1$, if $n\geq6$ & $\ell_{\varpi_1}^-(\varpi_n)$ & $r_1\dots r_{n-2}r_n$ \\
            &  $3$, if $n=5$ & $\ell^-_{\varpi_n}(\varpi_{n-1})$ & $r_5r_3r_4$ \\
\hline
${\bf E}_6$ & $5$ & $\ell^-_{\varpi_1}(\varpi_6)$ & $r_1r_3r_4r_5r_6$ \\
\hline
${\bf E}_7$ & $10$ & $\ell^-_{\varpi_7}(\varpi_7)$ & $r_7r_6r_5r_4r_2r_3r_4r_5r_6r_7$ \\
\hline
${\bf E}_8$ & $7\leq\ell_{{\bf E}_8}\leq 29=\ell^{\varpi_8}$ & $\ell^{\varpi_8}=\ell^-_{\varpi_8}(\varpi_j),\;\forall j$ & \\
\hline
${\bf F}_4$ & $8$ & $\ell^-_{\varpi_1}(\varpi_1)$ & $r_1r_2r_3r_2r_4r_3r_2r_1$ \\
\hline
${\bf G}_2$ & $3$ & $\ell^-_{\varpi_1}(\varpi_1)=\ell^-_{\varpi_2}(\varpi_2)$ & $r_1r_2r_1$, resp. $r_2r_1r_2$ \\
\hline
\end{tabular}
\end{theorem}

\begin{proof}
The values of $\ell^{\varpi_j}$ are computed and compared for each simple root system in respective propositions as follows. 

For type ${\bf A}_n$, see Proposition \ref{Prop An ell fundwei}.

Types ${\bf B}_n$ and ${\bf C}_n$ are dual to each other, and it suffices to consider one of them. We choose type ${\bf C}_n$. Some remarks on type ${\bf B}_n$ are given in \S\ref{Sect Bn}.

For type ${\bf C}_n$, see Proposition \ref{Prop Cn ell fundwei}.

For type ${\bf D}_n$, see Proposition \ref{Prop Dn ell fundwei}.

For type ${\bf E}_6$, see Proposition \ref{Prop E6 ell and LSD}.

For type ${\bf E}_7$, see Proposition \ref{Prop E7 ell}.

For type ${\bf E}_8$, see Proposition \ref{Prop E8 ell bounds}.

For type ${\bf F}_4$, see Proposition \ref{Prop F4 ell}.

For type ${\bf G}_2$, see Proposition \ref{Prop G2 ell}.
\end{proof}

\begin{theorem}\label{Theo Values ellsdDelta}
The values of the numerical invariant $\ell_{\Delta}^{\rm sd}$ for a simple root system $\Delta$ are as follows:

\begin{tabular}{|l|l|l|l|}
\hline
Type of $\Delta$ & $\ell_\Delta^{\rm sd}$ & Attained at $\ell_h^-(\lambda)$ & $w\in W(\ell_\Delta):\langle w\lambda,h\rangle<0$ \\
\hline
${\bf A}_n$ & $\lfloor\frac{n+2}{2}\rfloor$ & $\ell_{\rho}^-(\varpi_1)$ & $r_{\lfloor\frac{n+2}{2}\rfloor}...r_1$ \\
\hline
${\bf B}_n$, $n\geq 2$ & $n=\ell_{\Delta}$ & $\ell_{\varpi_1}^-(\varpi_n)$ & $r_1\dots r_n$ \\
            &     & & \\
\hline
${\bf C}_n$, $n\geq 2$ & $n=\ell_{\Delta}$ & $\ell_{\varpi_1}^-(\varpi_n)$ & $r_1\dots r_n$ \\
            &     & & \\
\hline
${\bf D}_n$, $n\geq 4$ & $n-1$ & $\ell_{\varpi_1}^-(\varpi_n)$ & $r_1\dots r_{n-2}r_n$ \\
            & $=\ell_\Delta$ if $n\ne5$ & & \\
\hline
${\bf E}_6$ & $9$ & $\ell^-_{\varpi_1}(\varpi_1+\varpi_6)$ & $r_1r_3r_4r_2r_5r_4r_3r_1r_6$ \\
\hline
${\bf E}_7$ & $10$ & $\ell^-_{\varpi_7}(\varpi_7)$ & $r_7r_6r_5r_4r_2r_3r_4r_5r_6r_7$ \\
\hline
${\bf E}_8$ & $\ell_{{\bf E}_8}$ & & \\
\hline
${\bf F}_4$ & $8=\ell_{\Delta}$ & $\ell^-_{\varpi_1}(\varpi_1)$ & $r_1r_2r_3r_2r_4r_3r_2r_1$ \\
\hline
${\bf G}_2$ & $3=\ell_{\Delta}$ & $\ell^-_{\varpi_1}(\varpi_1)=\ell^-_{\varpi_2}(\varpi_2)$ & $r_1r_2r_1$, resp. $r_2r_1r_2$ \\
\hline
\end{tabular}
\end{theorem}

\begin{proof}
For type ${\bf A}_n$ the statement is proven in Proposition \ref{Prop An ell hirootAndSelfdualh}.

For types ${\bf B}_n,{\bf C}_n,{\bf D}_{2m},{\bf E}_7,{\bf E}_8,{\bf F}_4,{\bf G}_2$, the property $w_0=-1$ holds and all weights are self-dual. We automatically have $\ell_\Delta=\ell_\Delta^{\rm sd}$ and the values are given in Theorem \ref{Theo Values ellDelta}.

For type ${\bf D}_{2m+1}$, the statement is proven is Proposition \ref{Prop Dn ell and LSD}.

For type ${\bf E}_6$, the result is given in Proposition \ref{Prop E6 ell and LSD}.
\end{proof}

\subsection{Type ${\bf A}_n$}\label{Sect An}

Let $e_1,...,e_{n+1}$ be the standard basis of $\RR^{n+1}$, considered with the standard scalar product so that the basis is orthonormal. We denote $\varepsilon = e_1+...+e_{n+1}$. In these coordinates the relevant data for the root system ${\bf A}_n$ is the following (cf. \cite{Bourbaki-Lie-2}).

\begin{enumerate}
\item[$\bullet$] Simple roots: $\alpha_j=e_j-e_{j+1}$, $j=1,...,n$.
\item[$\bullet$] Fundamental weights: $\varpi_j=b_j-\frac{j}{n+1}\varepsilon$, where $b_j=e_1+...+e_j$ for $j=1,...,n$.
\item[$\bullet$] Simple reflections: $r_1,...,r_n$, where $r_j$ transposes $e_j$ and $e_{j+1}$ and leaves the remaining basis vectors fixed.
\end{enumerate}


\begin{prop}\label{Prop An ell hPrince}
Let $h=\rho^\vee=\rho=\varpi_1 +...+\varpi_n$ be the principal semisimple element of $\mk{sl}_{n+1}$. Then
$$
\ell^h = \ell_h^-(\varpi_1) = \lfloor(n+2)/2\rfloor \;.
$$
Furthermore, there is a unique $w$ of length $\ell^h$ fulfilling the inequality $\langle w\varpi_1,h\rangle<0$; it is given by the product of the first $\ell^h$ simple reflections:
$$
w=r_{\ell^h}r_{\ell^h-1}...r_1 \;.
$$
\end{prop}

\begin{proof}
First note that, by Lemma \ref{Lemma lminush min bj}, the minimal value of $\ell_h^-$ is attained at some of the fundamental weights $\varpi_1,...,\varpi_{n}$. Moreover, since $h$ is self-dual, it suffices to consider $\varpi_1,...,\varpi_{\lfloor\frac{n+1}{2}\rfloor}$.

We now compute the value $\ell_h^-(\varpi_1)$ and show that it is equal to $l:=\lfloor(n+2)/2\rfloor$. We may compute with $e_1$ instead of $\varpi_1=e_1-\frac{1}{n+1}\varepsilon$. The Weyl group orbit of $e_1$ is the basis $e_1,...,e_{n+1}$. Note that $\langle w e_j,h\rangle<0$ if and only if $j\geq l$, so $l$ is the minimal index yielding a negative scalar product. The shortest element of the symmetric group sending $e_1$ to $e_l$ is the above $w$, and it has length $l$. So $\ell_h^-(\varpi_1)=l$.

It remains to show that for any fundamental weight $\varpi_j$ we have $\ell_h^-(\varpi_j)\geq l$. Observe that, for any $w\in W$ and any simple reflection $r_k$, the difference between $\langle w\varpi_j,h\rangle$ and $\langle r_k w\varpi_j,h\rangle$ is either $0$ or $\pm 2$. We have, for $j\in\{1,2,...,\frac{n+1}{2}\}$,
$$
\langle\varpi_j,h\rangle = n+(n-2)+...+(n-2j) = jn-j(j+1)=j(n-1-j) \;.
$$
So $\varpi_j$ needs to be moved by an element of length at least $j(n-1-j)/2$ in order for the scalar product with $h$ to become negative. This number is minimal for $j=1$. This proves the proposition.
\end{proof}

\begin{prop}\label{Prop An ell fundwei}
\begin{enumerate}
\item[{\rm (i)}] For $h=\varpi_1$, the value of $\ell^h$ is $1$ and is attained at $\lambda=\varpi_1$ for a unique Weyl group element, $r_1$.
\item[{\rm (ii)}] For $h=\varpi_1$ and $\lambda=\varpi_j$ with $j\in\{1,2,...,n\}$, the value of $\ell^-_{\varpi_1}(\varpi_j)$ is $j$ and is attained for the Weyl group element $r_1...r_j$.
\item[{\rm (iii)}] For $h=\varpi_j$ with $j\in\{1,...,\lfloor\frac{n+1}{2}\rfloor\}$, the value of $\ell^h$ is $j$ and it is attained at $\lambda=\varpi_1$ for a unique Weyl group element, $r_j...r_1$.
\end{enumerate}
\end{prop}

\begin{proof}
Statement (i) follows from (ii), which in turn follows from the fact that 
$$
W\varpi_1=\left\{e_k-\frac{1}{n+1}\varepsilon: k\in\{1,...,n+1\}\right\}=\{\varpi_1\}\cup\{r_k...r_1\varpi_1: k\in\{1,...,n\}\}\;,
$$
where notably $r_k...r_1$ is the shortest element $w\in W$ for which $w\varpi_1=e_k-\frac{1}{n+1}\varepsilon$. We have
$$
\langle e_k-\frac{1}{n+1}\varepsilon , \varpi_j\rangle = \langle e_k,e_1+...+e_j-\frac{j}{n+1}\varepsilon\rangle 
= \begin{cases} 1-\frac{j}{n+1} & ,\;\; {\rm if}\; k\leq j ; \\ -\frac{j}{n+1} & ,\;\; {\rm if}\; k<j . \end{cases} 
$$
This implies parts (ii) and (i).

In order to prove (iii), we fix $j\in\{1,...,\lfloor\frac{n+1}{2}\rfloor\}$ and we need to show that $\ell^-_{\varpi_j}(\varpi_k)\geq\ell^-_{\varpi_j}(\varpi_1)=j$ for all $k\in\{2,...,n\}$. We consider two cases: $k$ smaller, or greater, than $\lfloor\frac{n+1}{2}\rfloor$.

Case 1: $k\leq \lfloor\frac{n+1}{2}\rfloor$. Since $\ell^-_{\varpi_j}(\varpi_k)=\ell^-_{\varpi_k}(\varpi_j)$ we may assume that $k\leq j$ without loss of generality. We have, for any $w\in W$,
\begin{align*}
\langle w\varpi_k , \varpi_j\rangle &= \left\langle w(e_1+...+e_k-\frac{k}{n+1}\varepsilon) , \varpi_j\right\rangle = 
\left\langle w(e_1+...+e_k) , \varpi_j\right\rangle \\ 
&= \left\langle w(e_1+...+e_k) , \frac{n+1-j}{n+1}(e_1+...e_j)-\frac{j}{n+1}(e_{j+1}+...+e_{n+1})\right\rangle.
\end{align*}
Thus the value of $\langle w\varpi_k , \varpi_j\rangle$ depends solely on the cardinality of the set $M(w)=\{e_{j+1},...,e_{n+1}\}\cap w\{e_1,...,e_{k}\}$, say $m(w)=\# M(w)$. In order for $\langle w\varpi_k , \varpi_j\rangle$ to be negative, since $n+1-j\geq j$ by assumption, it is necessary that $m(w)> \frac{k}{2}$. The length of a Weyl group element satisfying this condition is minimized for $M(w)=\{e_{j+1},...,e_{j+\lfloor\frac{k+2}{2}\rfloor}\}$ and we have
$$
\min\left\{l(w): m(w)> \frac{k}{2}\right\} = l((r_{j}...r_{\lfloor\frac{k-1}{2}\rfloor})...(r_{j+\lfloor\frac{k+1}{2}\rfloor}...r_k)) =
\left\lfloor\frac{k+2}{2}\right\rfloor\left(j+1-\left\lfloor\frac{k-1}{2}\right\rfloor\right) .
$$
The assumption $k\leq j$ implies $j\leq\lfloor\frac{k+2}{2}\rfloor(j+1-\lfloor\frac{k-1}{2}\rfloor)$. Hence $\ell^-_{\varpi_j}(\varpi_k)\geq\ell^-_{\varpi_j}(\varpi_1)=j$.

Case 2: $k> \lfloor\frac{n+1}{2}\rfloor>j$. Since $\ell^-_{\varpi_j}(\varpi_k)=\ell^-_{\varpi_k^*}(\varpi_j^*)=\ell^-_{\varpi_{n+1-j}}(\varpi_{n+1-k})=\ell^-_{\varpi_{n+1-k}}(\varpi_{n+1-j})$ we may assume that $n+1-k\leq j$ without loss of generality. 
Put $k'=n+1-k$. We have, for any $w\in W$,
\begin{align*}
\langle w\varpi_k , \varpi_j\rangle &= \left\langle w\left(\frac{k'}{n+1}\varepsilon-e_{k+1}-...-e_{n+1}\right) , \varpi_j\right\rangle = 
\langle w(-e_{k+1}-...-e_{n+1}) , \varpi_j\rangle \\ 
&= -\left\langle w(e_{k+1}+...+e_{n+1}) , \frac{n+1-j}{n+1}(e_1+...e_j)-\frac{j}{n+1}(e_{j+1}+...+e_{n+1})\right\rangle.
\end{align*}
The value of $\langle w\varpi_k , \varpi_j\rangle$ depends solely on the cardinality of the set $M'(w)=\{e_{1},...,e_{j}\}\cap w\{e_{k+1},...,e_{n+1}\}$, put $m'(w)=\# M'(w)$. In order for $\langle w\varpi_k , \varpi_j\rangle$ to be negative, since $n+1-j\geq j$ by assumption, it is necessary that $m'(w)> \frac{k}{2}$. Consider $M''(w)=\{e_{1},...,e_{\lfloor\frac{n+1}{2}\rfloor}\}\cap w\{e_{k+1},...,e_{n+1}\}$ and $m''(w)=\#M''(w)$. We clearly have $M'(w)\subseteq M''(w)$ and so $m'(w)\leq m''(w)$. Hence
$$
\min\left\{l(w):m'(w)>\frac{k}{2} \right\} \geq \min\left\{l(w):m''(w)>\frac{k}{2} \right\} = \left\lfloor\frac{k'+2}{2}\right\rfloor\left(\left\lfloor\frac{n-1}{2}\right\rfloor+1-\left\lfloor\frac{k'-1}{2}\right\rfloor\right) \;,
$$
where the last equality is derived from Case 1. Also from Case 1 we derive $\lfloor\frac{k'+2}{2}\rfloor(\lfloor\frac{n-1}{2}\rfloor+1-\lfloor\frac{k'-1}{2}\rfloor)\geq \lfloor\frac{n-1}{2}\rfloor\geq j$. This completes the proof.
\end{proof}

\begin{prop}\label{Prop An ell hirootAndSelfdualh}
Let $h\in\Lambda_{\rm sd}^+$, so that $h$ has the form
$$
h=\sum\limits_{j=1}^{\lfloor\frac{n+1}{2}\rfloor} a_j(\varpi_j+\varpi_{n+1-j}) \quad,\quad with\; a_1,...,a_{\lfloor\frac{n}{2}\rfloor}\in \ZZ_{\geq 0}\; and, \;for\;odd\;n,\; a_{\lfloor\frac{n+1}{2}\rfloor}\in \frac12\ZZ_{\geq 0}.
$$
Then $\ell^h=\ell_h^-(\varpi_1)=n+1-j_h$, where $j_h=\max\{j:a_j\ne 0\}$, and this value is
attained for the Weyl group element $r_{n+1-j_h} ...r_2r_1$.
\end{prop}

\begin{proof}
Let $h$ be of the given form. In the basis $e_1,...,e_{n+1}$ we obtain
$$
h= \sum\limits_{j=1}^{j_h} b_j(e_j-e_{n+2-j}) \;,
$$
with $b_1\geq...\geq b_{j_h}\geq 1$. In terms of the $a_j$, we have $b_j=\sum\limits_{k=1}^{j} a_k$ for $1\leq j\leq \lfloor\frac{n}{2}\rfloor$ and, if $n+1=2m$ is even, $b_m=b_{m-1}+2a_m$.

Since, $\varpi_k=\varpi_{n+1-k}^*$ and $h=h^*$, it suffices to determine, or estimate, the numbers $\ell^-_h(\varpi_k)$ for $1\leq k\leq \lfloor\frac{n+1}{2}\rfloor$. As far as the scalar products are concerned, we can compute with $e_1+...+e_k$ instead of $\varpi_k$. We observe the following: if $\langle w\varpi_k,h\rangle<0$, then $\#M(w)>\min\{\frac{k}{2},\frac{j_h}{2}\}$.

For $k=1$, we obtain
$$
\ell^-_h(\varpi_1)=l(r_{n+1-j_h} ...r_2r_1) = n+1-j_h \;.
$$
For $k\geq 2$, an argument analogous to the one used in the proof of Proposition \ref{Prop An ell fundwei}, applied here for the set $M(w)=\{e_{n+2-j_h},...,e_{n+1}\}\cap w\{e_1,...,e_k\}$, shows that $\langle w\varpi_k,h\rangle<0$ implies $l(w)>n+1-j_h$. Thus $\ell^-_h(\varpi_k)>\ell^-_h(\varpi_1)$ and $\ell^h=\ell^-_h(\varpi_1)$.
\end{proof}

\begin{rem}
As a particular case of Example \ref{Exa shortroot}, for $h=\varpi_1+\varpi_n=e_1-e_n$, the highest root, we obtain the value $\ell^h=n$, and it is attained at every nonzero dominant weight $\lambda$ for a suitable Weyl group element. All such Weyl group elements are Coxeter elements.
\end{rem}

\begin{prop}\label{Lemma lminusAn}
\begin{enumerate}
\item[{\rm (i)}] The value of $\ell_{{\bf A}_n}$ is $1$ and is attained for $h=\varpi_1$, $\lambda=\varpi_1$ and $w=r_1$.
\item[{\rm (ii)}] The value of $\ell_{{\bf A}_n}^{\rm sd}$ is $\lfloor(n+2)/2\rfloor$ and is attained at $h=\varpi_1+...+\varpi_n$, the principal semisimple element, $\lambda=\varpi_1$, the first fundamental weight, and $w=r_{\lfloor(n+2)/2\rfloor} r_{\lfloor(n+2)/2\rfloor-1}... r_1$.
\end{enumerate}
\end{prop}

\begin{proof}
Part (i) follows from Proposition \ref{Prop An ell fundwei}.
Part (ii) follows from Proposition \ref{Prop An ell hirootAndSelfdualh} and Proposition \ref{Prop An ell hPrince}.
\end{proof}

\subsection{Type ${\bf C}_n$}\label{Sect Cn}

Let $e_1,...,e_n$ be the standard basis of $\RR^n$, considered with the standard scalar product so that the basis is orthonormal. In these coordinates the relevant data for the root system ${\bf C}_n$ is the following (cf. \cite{Bourbaki-Lie-2}).

\begin{enumerate}
\item[$\bullet$] Simple roots: $\alpha_1=e_1-e_2, \alpha_2=e_2-e_3,...,\alpha_{n-1}=e_{n-1}-e_{n}, \alpha_n=2e_n$.
\item[$\bullet$] Fundamental weights: $\varpi_j=e_1+...+e_j$ for $j=1,...,n$.
\item[$\bullet$] Simple reflections: $r_1,...,r_n$, where, for $j=1,...,n-1$, $r_j$ transposes $e_j$ and $e_{j+1}$, and $r_n$ negates $e_n$. The remaining basis vectors are left fixed in either case.
\end{enumerate}


%
%
%

\begin{prop}\label{Prop Cn ell fundwei}
\begin{enumerate}
\item[{\rm (i)}] For $h=\varpi_1=e_1$, the value of $\ell^h$ is $n$ and is attained at $\lambda=\varpi_n=e_1+...+e_n$ for a unique Weyl group element, $r_1r_2...r_n$.
\item[{\rm (ii)}] For $h=\varpi_j$ with $j\in\{2,...,n\}$, the value of $\ell^{h}$ is $2n-j$ and is attained at $\lambda=\varpi_1$ for a unique Weyl group element $r_j\dots r_{n}r_{n-1}\dots r_1$.
\end{enumerate}
\end{prop}

\begin{proof}
We begin with $h=\varpi_1=e_1$. We shall compute $\ell^-_{\varpi_1}(\varpi_j)$ for all $j$. The Weyl group orbit of our element is $W\varpi_1=We_1=\{\pm e_1,...\pm e_n\}$. For $w\in W$ with $l(w)<n$ we have $w\varpi_1\in\{e_1,...,e_{n-1}\}$ and hence $\langle \varpi_j,w\varpi_1\rangle\geq0$ for all $j$. Thus $\ell^-_{\varpi_1}(\varpi_j)\geq n$ for all $j$. The element $w_1=r_n...r_1$ is the unique $w$ with $l(w)=n$ for which $w\varpi_1$ has a negative sign. We have $\langle \varpi_n,w_1\varpi_1\rangle=\langle e_1+...+e_n,-e_n\rangle=-1<0$ and $\langle \varpi_j,w_1\varpi_1\rangle=0$ for $j<n$, hence
$$
\ell^{\varpi_1}=\ell^-_{\varpi_1}(\varpi_n)=l(r_1...r_n)=n \;.
$$
This already proves part (i). We continue with the computation of the values $\ell^-_{\varpi_1}(\varpi_j)$ for $j<n$, as they are relevant the the remaining parts due to the symmetry. As established, the shortest element $w$ such that $w\varpi_1\in\{-e_1,...,-e_n\}$ is $w_1$. The shortest element $w$ for which $ww_1\varpi_1$ is not orthogonal to $\varpi_j$ is $w_2=r_j...r_{n-1}$. We deduce that $w_{j,1}=w_2w_1$ is the shortest element with $\langle \varpi_j,w_{j,1}\varpi_1\rangle<0$ and obtain
$$
\ell^-_{\varpi_1}(\varpi_j)=l(r_1...r_nr_{n-1}...r_j) = n+n-j=2n-j \;, \;\;{\rm for}\;\; j=1,...,n.
$$

Since $\ell^-_{\varpi_j}(\varpi_1)=\ell^-_{\varpi_1}(\varpi_j)$, we can deduce that $\ell^{\varpi_j}\leq 2n-j$. We shall show that equality holds for $j\geq 2$.

Next, we address $h=\varpi_n=e_1+...+e_n$ and compute the numbers $\ell^-_{\varpi_n}(\varpi_j)$ for $j\geq 2$. Given $w\in W$, in order for $\langle\varpi_j,w\varpi_n\rangle$ to be negative, $w$ needs to turn into $-1$ at least $i=\lfloor\frac{j+2}{2}\rfloor$ of the coefficients of $e_1,...,e_j$. In order to minimize the length, the signs of $e_{\lfloor\frac{j+1}{2}\rfloor},...,e_j$ ought to be negated. This is achieved in a minimal way by the element $(r_j...r_n)(r_{j-1}...r_n)...(r_{\lfloor\frac{j+1}{2}\rfloor}...r_n)$. Thus
\begin{align*}
\ell^-_{\varpi_n}(\varpi_j)&=l((r_j...r_n)(r_{j-1}...r_n)...(r_{\lfloor\frac{j+1}{2}\rfloor}...r_n))\\
 &=(n-j+1)+(n-j+2)+...+(n-j+i) = (n-j)i + \frac12 i(i+1) \;.
\end{align*}

For $j=2$, we get $\ell^-_{\varpi_n}(\varpi_2)=(n-2)2 + 3 = 2n-1> 2n-2 = \ell^-_{\varpi_1}(\varpi_2)$.

For $j=3$, we get $\ell^-_{\varpi_n}(\varpi_3)=(n-3)2 + 3 = 2n-3 = \ell^-_{\varpi_1}(\varpi_3)$.

For $j=4$, we get $\ell^-_{\varpi_n}(\varpi_4)=(n-4)3 + 6 = 2n+6 > 2n-4 = \ell^-_{\varpi_1}(\varpi_4)$.

Since $\ell^-_{\varpi_n}(\varpi_j)$ is quadratic in $j$, we deduce that $\ell^-_{\varpi_n}(\varpi_j)\geq\ell^-_{\varpi_1}(\varpi_j)$ for $j\in\{2,...,n\}$, with equality holding if and only if $j=3$. We have also shown that $\ell^-_{\varpi_1}(\varpi_j)\geq\ell^-_{\varpi_1}(\varpi_n)$, with equality holding if and only if $j=n$. It follows that
$$
\ell^{\varpi_n} = \ell^{-}_{\varpi_n}(\varpi_1) = l(r_n...r_1) = n \;.
$$
This proves part (ii) for $j=n$.

We now consider the remaining case $h=\varpi_k$ with $k\in\{2,...,n-1\}$, and estimate the numbers $\ell^-_{\varpi_k}(\varpi_j)$ for $j\in\{2,...,n-1\}$. Since $\ell^-_{\varpi_k}(\varpi_j)=\ell^-_{\varpi_j}(\varpi_k)$ we may assume that $j\leq k$. We shall prove the inequality $\ell^-_{\varpi_j}(\varpi_1)<\ell^-_{\varpi_j}(\varpi_k)$ by comparing the trajectories of basis vectors under any Weyl groups elements $w_{k,1},w_{k,j}\in W$ whose lengths give the respective values of $\ell^-_{\varpi_k}$. The element $w_{k,1}$ was constructed above. For any $w$, the values of the scalar products $\langle w\varpi_1,\varpi_k\rangle$ and $\langle w\varpi_j,\varpi_k\rangle$ are determined by the respective intersections $\{\pm e_1,...,\pm e_k\}\cap w\{e_1\}$ and $\{\pm e_1,...,\pm e_k\}\cap w\{e_1,...,e_j\}$. In order for either scalar product to be negative, the $w$-trajectory of some elements $e\in\{e_1,...,e_j\}$ must fulfill the following three steps: (a) leave the set $\{e_1,...,e_k\}$ by reaching $\{e_{k+1}\}$; (b) acquire a negative sign by passing through $\{-e_n\}$; (c) return to $\{-e_1,...,-e_k\}$. In both cases steps (a),(b),(c) have to be performed for at least one $e$, and in the case of $\varpi_1=e_1$ this suffices. Thus it suffices to show that necessary performances of step (a) require more length for $\varpi_j$ than for $\varpi_1$. For $\varpi_1$, this length is $k=l(r_k...r_1)$. The inequality $\langle w_{k,j}\varpi_j,\varpi_k\rangle<0$ means that the set $\{\pm e_1,...,\pm e_k\}\cap w_{k,j}\{e_1,...,e_j\}$ has strictly more negative than positive signs. Since $k<n$, it follows that the $w_{k,j}$-trajectories of at least $i=\lfloor\frac{j+2}{2}\rfloor$ elements of the set $\{e_1,...,e_j\}$ must leave the set $\{e_1,...,e_k\}$. The required number of simple transpositions is $(k-j+1)+(k-j+2)+...+(k-j+i) = (k-j)i+\frac12 i(i+1)$. An elementary calculation yields $k\leq (k-j)i+\frac12 i(i+1)$. Furthermore, equality holds if and only if $j=k=3$. Thus we obtain
$$
\ell^-_{\varpi_j}(\varpi_1)<\ell^-_{\varpi_j}(\varpi_k) \quad,\quad \forall j,k\in\{2,...,n-1\} \;.
$$
Summarizing the above observations on $\ell^-_{\varpi_j}$, we obtain
$$
\ell^{\varpi_j}=\ell^-_{\varpi_j}(\varpi_1) = l(r_j...r_n...r_1) = 2n-j \;,
$$
the value also being attained, for the special case $j=3$, at $\ell^-_{\varpi_3}(\varpi_n)=l((r_n...r_2)(r_n...r_3))$. This completes the proof of the proposition.
\end{proof}

\begin{coro}\label{Coro Cn ell}
The value of $\ell_{{\bf C}_n}$ is $n$ and is attained for $h=\varpi_1$, $\lambda=\varpi_n$, $w=r_1...r_n$ as well as for the dual combination $h=\varpi_n$, $\lambda=\varpi_1$, $w=r_n...r_1$.
\end{coro}

\begin{prop}\label{Prop Cn ell hPrince}
For $h=\rho=\varpi_1+...+\varpi_n=ne_1+(n-1)e_2+...+2e_{n-1}+e_n$, or $h=\rho^\vee=\varpi_1^\vee+...+\varpi_n^\vee=\frac12((2n-1)e_1+(2n-3)e_2+...+3e_{n-1}+e_n)$, the value of $\ell^\rho$ is $n$ and is attained at $\lambda=\varpi_1=e_1$ for a unique Weyl group element, $r_nr_{n-1}...r_1$.
\end{prop}

\begin{proof}
The computations are the same as for $h=\varpi_n$.
\end{proof}

\subsection{Type ${\bf B}_n$}\label{Sect Bn}

According to Remark \ref{Rem h in Lambda} and Corollary \ref{Coro ellDeltaDual}, since the simple root systems of type ${\bf B}_n$ and ${\bf C}_n$ are dual to each other, and $w_0=-1$ holds in either case, we have $\ell_{{\bf B}_n}=\ell_{{\bf C}_n}=\ell^{\rm sd}_{{\bf B}_n}=\ell^{\rm sd}_{{\bf C}_n}$. The individual values at the fundamental weights $\varpi_j$ of ${\bf B}_n$ are given by $\ell^{\varpi_j}=\ell^{\varpi_j^\vee}$, where $\varpi_j^\vee$ is the $j$-th fundamental weight of ${\bf C}_n$. The values of $\ell^{\varpi_j^\vee}$ are given in Proposition \ref{Prop Cn ell fundwei}. The value of $\ell^h$ for the principal semisimple element $h=\rho^\vee$ of ${\bf B}_n$ is also computed in Proposition \ref{Prop Cn ell hPrince}. With this the collected data for ${\bf B}_n$ is complete.

\subsection{Type ${\bf D}_n$}\label{Sect Dn}

Let $n\geq 4$. Let $e_1,...,e_n$ be the standard orthonormal basis of $\RR^n$ endowed with the standard scalar product. In these coordinates the relevant data for the root system ${\bf C}_n$ is the following (cf. \cite{Bourbaki-Lie-2}).

\begin{enumerate}
\item[$\bullet$] Simple roots: $\alpha_1=e_1-e_2$, $\alpha_2=e_2-e_3,...,$ $\alpha_{n-1}=e_{n-1}-e_{n}$, $\alpha_n=2e_n$.
\item[$\bullet$] Fundamental weights: $\varpi_j=e_1+...+e_j$ for $j=1,...,n-2$,
$$
\varpi_{n-1} = \frac12 (e_1+e_2+...+e_{n-1}-e_{l}) \;,\; \varpi_l = \frac12 (e_1+e_2+...+e_{n-1}+e_{n}) \;.
$$
\item[$\bullet$] Simple reflections: $r_1,...,r_n$, where for $j=1,...,n-1$, $r_j$ transposes $e_j$ and $e_{j+1}$, and $r_n$ negates and transposes $e_{n-1}$ and $e_n$. The remaining basis vectors are left fixed in either case. 
\item[$\bullet$] $\Lambda^+_{\rm sd}=\{b_1\varpi_1+...+b_{n-2}\varpi_{n-2}+b_{n-1}(\varpi_{n-1}+\varpi_n): b_j\in \ZZ_{\geq0}\}=\{a_1e_1+...+a_{n-1}e_{n-1}: a_j\in \ZZ, a_1\geq a_2\geq ... \geq a_{n-1}\geq 0\}$.
\end{enumerate}


%
%
%

\begin{prop}\label{Prop Dn ell fundwei}
\begin{enumerate}
\item[{\rm (i)}] For $h=\varpi_1$, the value of $\ell^h$ is $n-1$ and is attained at two fundamental weights $\lambda=\varpi_{n-1}$ and $\varpi_{n}$ for the Weyl group elements $r_1...r_{n-1}$ and $r_1...r_{n-2}r_n$, respectively.
\item[{\rm (ii)}] For $h=\varpi_2$, the value of $\ell^h$ is $2n-3$ and is attained at three fundamental weights $\lambda=\varpi_1,\varpi_{n-1}$ and $\varpi_n$ for the Weyl group elements $r_{2}...r_{n-2}r_n...r_1$, $r_2...r_{n-2}r_nr_1...r_{n-1}$ and $r_2...r_{n-1}r_1...r_{n-2}r_{n}$, respectively.
\item[{\rm (iii)}] For $h=\varpi_3$, the value of $\ell^h$ is $2n-5$ and is attained at two fundamental weights $\lambda=\varpi_{n-1}$ and $\varpi_n$ for the Weyl group elements $r_3...r_{n-2}r_nr_2...r_{n-1}$ and $r_3...r_{n-1}r_2...r_{n-2}r_{n}$, respectively.
\item[{\rm (iv)}] If $n\geq 6$, then, for $h=\varpi_j$ with $j\in\{4,...,n-2\}$, the value of $\ell^h$ is $2n-j-1$ and is attained at $\lambda=\varpi_1=e_1$ for the Weyl group element $r_{j}...r_{n-2}r_n...r_1$.
\item[{\rm (v)}] If $n=5$, then
$$
3=\ell^{\varpi_5}=\ell^-_{\varpi_5}(\varpi_4)=l(r_5r_3r_4)=l(r_4r_3r_5)=\ell^-_{\varpi_4}(\varpi_5)=\ell^{\varpi_4}\;.
$$
\item[{\rm (vi)}] If $n\ne 5$, then, for $h=\varpi_{n-1}$, resp. $\varpi_n$, the value of $\ell^h$ is $n-1$ and is attained at $\lambda=\varpi_1$ for the Weyl group element $r_{n-1}...r_1$, resp. $r_nr_{n-2}...r_1$.
\end{enumerate}
\end{prop}

\begin{proof}
First recall that $\varpi_n$ and $\varpi_{n-1}$ are related by an automorphism of the root system of ${\bf D}_n$ fixing all other fundamental weights. This automorphism is inner for even $n$, so that all weights are self-dual, and outer for odd $n$, so that $\varpi_n^*=\varpi_{n-1}$. In either case, we obtain $\ell^-_{\varpi_j}(\varpi_n)=\ell^-_{\varpi_j}(\varpi_{n-1})$ for all $j\in\{1,...,n-2\}$ and $\ell^{\varpi_n}=\ell^{\varpi_{n-1}}$.

To prove (i) we take $h=\varpi_1=e_1$ and compute the numbers $\ell^-_{\varpi_1}(\varpi_j)=\ell_{\varpi_j}(\varpi_1)$ for $j=1,...,n$. It is convenient to act on $\varpi_1$, and construct elements $w_{j,1}$ of minimal length among those satisfying $\langle \varpi_j,w\varpi_1\rangle<0$. Then $w_{1,j}$ can be taken to be $w_{j,1}^{-1}$. For any $w\in W$ with $l(w)<n-1$, we have $w\varpi_1\in\{e_1,...,e_{n-1}\}$ and hence $\langle \varpi_j,w\varpi_1\rangle>0$ for all $j$. On the other hand, the element $w_{n,1}=r_nr_{n-2}...r_1$ yields $\langle \varpi_{n},w_{n,1}\varpi_1\rangle=\frac12\langle e_1+...+e_{n-1} +e_n,-e_n\rangle=-\frac12<0$. This immediately implies part (i):
$$
\ell^{\varpi_1}=\ell^-_{\varpi_1}(\varpi_{n}) = l(r_1...r_{n-2}r_n) = n-1\;.
$$
For $j\in\{1,...,n-2\}$, in order to obtain a negative value for $\langle \varpi_j,w\varpi_1 \rangle$ we first need to apply $w_{n,1}$, the shortest element sending $\varpi_1$ to $\{-e_1,...,-e_n\}$, and compose it with the shortest element sending $w_{n,1}\varpi_1=-e_n$ to $\{-e_1,...,-e_j\}$, which is $r_j...r_{n-1}$. Put $w_{j,1}=r_j...r_{n-1}r_nr_{n-2}...r_1$. We obtain
\begin{gather}\label{For Dn ell 1j}
\ell^-_{\varpi_1}(\varpi_j)=\ell^-_{\varpi_j}(\varpi_1) = l(w_{j,1}) =2n-j-1 \;.
\end{gather}

Next, we address $h=\varpi_j$ and the number $\ell^{\varpi_j}$ for $j\in\{2,...,n-2\}$. We have computed $\ell^-_{\varpi_j}(\varpi_1)$ above. For $k\in\{2,...,n-2\}$, the inequality $\ell^-_{\varpi_j}(\varpi_1)<\ell^-_{\varpi_j}(\varpi_k)$ holds, by the same argument as in type ${\bf C}_n$, the proof of Proposition \ref{Prop Cn ell fundwei}. 

To compute $\ell^{\varpi_j}$ it remains to consider $\lambda=\varpi_n=\frac12(e_1+...+e_n)$ and compute $\ell^-_{\varpi_j}(\varpi_n)$. For $w\in W$, the negativity of the scalar product $\langle w\varpi_n,\varpi_j\rangle$ is equivalent to the set $\{-e_1,...,-e_j\}\cap w\{e_1,...,e_n\}$ having cardinality at least $\lfloor\frac{j+2}{2}\rfloor$. Negative signs are produced in pairs by $r_n$ and in order for more negative signs to be produced, both $e_{n-1},e_n$ have to be removed from the set $\{e_{n-1},e_n\}$ before $r_n$ is applied again. The building blocks for the following construction are the elements $w_i=r_i...r_{n-1}r_{i-1}...r_{n-2}r_n$, for $i\in\{2,...,n-2\}$. The element $w_i$ is an element of minimal length sending the pair $\{e_{n-1},e_n\}$ to the pair $\{-e_i,-e_{i-1}\}$. The length of this element is $2(n-i)+1$. Our task of obtaining a negative value for $\langle w\varpi_n,\varpi_j\rangle$ is achieved in a minimal way by the element
$$
w_{j,n} = \begin{cases} w_j w_{j-2}...w_{\lfloor\frac{j+1}{2}\rfloor+1} & \;,\;\; \textrm{if $\lfloor\frac{j+2}{2}\rfloor$ is even};\\ 
(r_{j}...r_{n-2}r_n)w_{j-1}w_{j-3}...w_{\lfloor\frac{j+1}{2}\rfloor+1} & \;,\;\; \textrm{if $\lfloor\frac{j+2}{2}\rfloor$ is odd}.
\end{cases}
$$
We consider the two cases separately.

Case 1: $\lfloor\frac{j+2}{2}\rfloor=2m$, $m\in\NN$. We obtain
\begin{align*}
\ell^-_{\varpi_j}(\varpi_n)=l(w_{j,n}) & = (2(n-j)+1)+(2(n-j+2)+1)+...+(2(n-j+2(m-1))+1) \\
 & = 2(n-j)m + 2(m-1)m + m \\
 & = 2(n-j-1)m + 2m^2 \;.
\end{align*}
We compare this number to $\ell^-_{\varpi_j}(\varpi_1)=2n-j-1$. We obtain
\begin{align*}
\ell^-_{\varpi_j}(\varpi_n)-\ell^-_{\varpi_j}(\varpi_1) & = 2(n-j-1)m + 2m^2 - (2n-j-1) \\
 & = 2-j + 2(n-j-1)(m-1) + 2m^2 \;. 
\end{align*}
An elementary calculation shows that
\begin{align*}
\ell^-_{\varpi_j}(\varpi_n)=\ell^-_{\varpi_j}(\varpi_1) & \tst\; j=2 \;,\\
\ell^-_{\varpi_j}(\varpi_n)<\ell^-_{\varpi_j}(\varpi_1) & \tst\; j=3 \;,\\
\ell^-_{\varpi_j}(\varpi_n)>\ell^-_{\varpi_j}(\varpi_1) & \tst\; j\geq4 \;.\\
\end{align*}

Case 1: $\lfloor\frac{j+2}{2}\rfloor=2m+1$, $m\in\NN$. We obtain
\begin{align*}
\ell^-_{\varpi_j}(\varpi_n)=l(w_{j,n}) & = n-j+(2(n-j)+1)+(2(n-j+2)+1)+...+(2(n-j+2m)+1) \\
 & = (n-j)(2m+1) + 2m(m+1) + m \\
 & = 2m^2+2(n-j+3)+n-j \;.
\end{align*}
We compare this number to $\ell^-_{\varpi_j}(\varpi_1)=2n-j-1$ and obtain
\begin{align*}
\ell^-_{\varpi_j}(\varpi_n)-\ell^-_{\varpi_j}(\varpi_1) & = 2m^2+2(n-j+3)+n-j - (2n-j-1) \\
 & = 2m^2 + m(n-j+3)-(n-1) \;.
\end{align*}
An elementary calculation shows that $\ell^-_{\varpi_j}(\varpi_n)>\ell^-_{\varpi_j}(\varpi_1)$ for all $j\geq 4$.

We conclude that the value of $\ell^{\varpi_j}$ is given by
\begin{align*}
& \ell^{\varpi_2}=\ell^-_{\varpi_2}(\varpi_1)=\ell^-_{\varpi_2}(\varpi_n)=\ell^-_{\varpi_2}(\varpi_{n-1}) = 2n-3 \;, \\
& \ell^{\varpi_3}=\ell^-_{\varpi_2}(\varpi_n)=\ell^-_{\varpi_2}(\varpi_{n-1}) = 2n-5 \;, \\
& \ell^{\varpi_j}=\ell^-_{\varpi_j}(\varpi_1) = 2n-j-1 \;,\;\;\forall j\in\{4,...,n-2\} \;,
\end{align*}
with respective Weyl group elements $w_{j,1}$ and $w_{j,n}$ defined along the above calculation. This completes the proofs of (ii),(iii) and (iv).

It remains to consider $h=\varpi_n$ and $\lambda\in\{\varpi_{n-1},\varpi_n\}$. We begin with $\lambda=\varpi_{n-1}$. The value of $\ell^-_{\varpi_n}(\varpi_{n-1})$ can be computed by the same method as applied above for $\ell^-_{\varpi_j}(\varpi_n)$. In order for an element $w\in W$ to yield a negative scalar product $\langle w\varpi_{n-1},\varpi_n\rangle$, it is necessary and sufficient that the set $\{-e_{1},...,-e_{n}\}\cap w\{e_1,...,e_{n-1},-e_n\}$ has cardinality at least $\lfloor\frac{n+2}{2}\rfloor$. The task is achieved in a minimal way by the element
$$
w_{n,n-1} = \begin{cases} r_n w_{n-2} w_{n-4}...w_{\lceil\frac{n}{2}\rceil+2}(r_{\lceil\frac{n}{2}\rceil}...r_{n-1}) & \;,\;\; \textrm{if $\lfloor\frac{n}{2}\rfloor$ is even};\\ 
r_n w_{n-2} w_{n-4}...w_{\lceil\frac{n}{2}\rceil+1}(r_{\lceil\frac{n}{2}-1\rceil}...r_{n-1}) & \;,\;\; \textrm{if $\lfloor\frac{n}{2}\rfloor$ is odd}.
\end{cases}
$$
The length of each of the above elements is equal to the sum of the lengths of the factors in the given product. The respective lengths can be computed in the same way as in the cases above. To compute the value of $\ell^{\varpi_n}$ a comparison to $\ell^-_{\varpi_n}(\varpi_1)=n-1$ suffices. We have
\begin{align*}
l(w_{n,n-1})\geq l(r_n)+l(w_{\lceil\frac{n}{2}\rceil+2})+l(r_{\lceil\frac{n}{2}\rceil}...r_{n-1}) & = 1+(2(n-\lceil\frac{n}{2}\rceil-2)+1)+(n-\lceil\frac{n}{2}\rceil) \\
& = 2(n-1)+\lfloor\frac{n}{2}\rfloor-2\lceil\frac{n}{2}\rceil \;.
\end{align*}
It follows that
$$
\ell^-_{\varpi_n}(\varpi_{n-1})-\ell^-_{\varpi_n}(\varpi_1)= l(w_{n,n-1})-l(w_{n,1})\geq n-1+\lfloor\frac{n}{2}\rfloor-2\lceil\frac{n}{2}\rceil \; >\; 0 \quad{\rm for}\quad n\geq 6 \;.
$$
For $n=5$, we obtain $w_{5,4}=r_5r_3r_4$, whence $\ell^-_{\varpi_5}(\varpi_{4})=3$. On the other hand, we have $\ell^-_{\varpi_5}(\varpi_{1})=l(r_5r_3r_2r_1)=4$ and $\ell^-_{\varpi_5}(\varpi_{5})=l(r_5r_3r_4r_2r_3r_5)=6$. Thus $\ell^{\varpi_5}=\ell^-_{\varpi_5}(\varpi_{4})=3$.

Next, we turn our attention to $\ell^-_{\varpi_n}(\varpi_n)$ which was computed so far only for $n=5$. Considerations analogous to the above bring us to the element
$$
w_{n,n} = \begin{cases} r_n w_{n-2} w_{n-4}...w_{\lceil\frac{n+2}{2}\rceil} & \;,\;\; \textrm{if $\lfloor\frac{n+2}{2}\rfloor$ is even};\\ 
r_n w_{n-2} w_{n-4}...w_{\lceil\frac{n}{2}\rceil} & \;,\;\; \textrm{if $\lfloor\frac{n+2}{2}\rfloor$ is odd}.
\end{cases}
$$
For $4\leq n\leq 7$, we have $w_{n,n}=r_nw_{n-2}$ and 
$$
\ell^-_{\varpi_n}(\varpi_n)=l(w_{n,n})=1+2(n-(n-2))+1=6\geq n-1=\ell^-_{\varpi_n}(\varpi_1) \;,
$$
with equality holding if and only if $n=7$. For $n\geq8$, our expression for $w_{n,n}$ contains a product of the form $w_{i+2}w_i$. To compare $\ell^-_{\varpi_n}(\varpi_n)$ to $\ell^-_{\varpi_n}(\varpi_1)=n-1$ it suffices to consider the length of $w_{i+2}w_i$. We have
$$
l(w_{i+2}w_i) = 2(n-i-2)+1+2(n-i)+1 = 4(n-i) - 2 \;.
$$
Hence
$$
l(w_{i+2}w_i) - l(w_{n,1}) = 4(n-i)-2-(n-1)=3n-4i-1 \;,
$$
which is positive if and only if $i<\frac34 n - \frac14$. Such $i$ occur in our expression and hence, for $n\geq 8$, $l(w_{n,n})>l(w_{i+2}w_i)>l(w_{n,1})$. 

We deduce that
$$
\ell^-_{\varpi_n}(\varpi_n)\geq n-1 = \ell^-_{\varpi_n}(\varpi_1)\;,
$$
with equality holding for a single special case at $n=7$.

Summarizing our results for $\ell^-_{\varpi_n}(\varpi_j)$ with $j=1,...,n$, we conclude that
$$
\ell^{\varpi_n}=\begin{cases} n-1 = \ell^-_{\varpi_n}(\varpi_1) = l(r_nr_{n-2}r_{n-3}...r_1) &,\; {\rm for}\; n\ne 5 \quad (=\ell^-_{\varpi_n}(\varpi_n) \;\;{\rm for}\; n=7)\;, \\
3=\ell^-_{\varpi_5}(\varpi_4)=l(r_5r_3r_4) &,\; {\rm for}\; n=5\;. \end{cases}
$$
This completes the proof.
\end{proof}

\begin{prop}\label{Prop Dn ell hPrince}
For $h=\rho=\varpi_1+...+\varpi_n=(n-1)e_1+(n-2)e_2+...+e_{n-1}$ the value of $\ell^\rho$ is $n$ and is attained at $\lambda=\varpi_1=e_1$ for two Weyl group elements $r_n...r_1$ and $r_{n-1}r_nr_{n-2}...r_1$.
\end{prop}

\begin{proof}
We employ methods analogous to the ones in Proposition \ref{Prop Dn ell fundwei} and use the notation from its proof.

If $w$ is either $r_n...r_1$ or $r_{n-1}r_nr_{n-2}...r_1$, then $w\varpi_1=-e_{n-1}$ and hence $\langle w\varpi_1,h\rangle<0$. The length of $w$ is minimal among the elements yielding a negative scalar product, because the reasoning applied in the proof of Proposition \ref{Prop Dn ell fundwei} to the case $\ell^-_{\varpi_1}(\varpi_j)$ with $j\in\{2,...,n-2\}$ extends without alteration for $j=n-1$ and this covers our $h$. Thus $\ell^-_h(\varpi_1)=n$.

To compare $\ell^-_h(\varpi_1)$ to the numbers $\ell^-_h(\varpi_j)$ with $j\geq 2$, we observe that, the expression for $h$ implies $\ell^-_h(\varpi_j)\geq\ell^-_{\varpi_n}(\varpi_j)$ for all $j$. Thus we may apply the results of Proposition \ref{Prop Dn ell fundwei},(v),(vi), which, along with $\ell^-_h(\varpi_1)=n$, imply that the value of $\ell^h$ is attained at $\varpi_1$, with a possible exception in the case $n=5$. A simple computation for $n=5$ yields
$$
\ell^-_h(\varpi_5)=l(r_3r_4r_1r_2r_3r_5)=l(r_5r_3r_4r_2r_3r_5)=6 \;.
$$
Thus $\ell^-_h(\varpi_1)<\ell^-_h(\varpi_5)$ in this case as well.
\end{proof}

\begin{prop}\label{Prop Dn ell and LSD}
The value of  $\ell_{{\bf D}_n}^{\rm sd}$ is
$$
\ell_{{\bf D}_n}^{\rm sd}=\ell{^\varpi_1}=\ell^-_{\varpi_1}(\varpi_n)=l(r_1...r_{n-2}r_n)=n-1 \;.
$$
For $n\ne 5$, the values of $\ell_{{\bf D}_n}$ and $\ell_{{\bf D}_n}^{\rm sd}$ coincide.

For $n=5$, the value of $\ell_{{\bf D}_5}$ is $\ell_{{\bf D}_5}=\ell^{\varpi_5}=\ell^-_{\varpi_5}(\varpi_4)=3$.
\end{prop}

\begin{proof}
As shown in Proposition \ref{Prop Dn ell fundwei}, for $n=5$, the minimum among $\ell^{\varpi_j}$ for $j=1,...,n$ is attained for $\varpi_1$ and $\ell{^\varpi_1}=\ell^-_{\varpi_1}(\varpi_n)=l(r_1...r_{n-2}r_n)=n-1$. Since $\varpi_1\in\Lambda^+_{\rm sd}$, we obtain $\ell_{{\bf D}_n}^{\rm sd}=\ell{^\varpi_1}=n-1$, due to Corollary \ref{Coro ellDeltaDual}.

For $n=5$, the values of $\ell^{\varpi_j}$, $j=1,...,5$, are $4,7,5,3,3$, respectively. Thus $\ell_{{\bf D}_5}=\ell^{\varpi_4}=\ell^{\varpi_5}=3$. Since $\varpi_4^*=\varpi_5\notin \Lambda^+_{\rm sd}$ this does not yield the value of $\ell_{{\bf D}_5}$. The monoid $\Lambda^+_{\rm sd}$ is generated by $\varpi_1,\varpi_2,\varpi_3,\varpi_4+\varpi_5$. The values of the first three generators are known. It remains to consider $h=\varpi_4+\varpi_5=e_1+e_2+e_3+e_4$. The same computation as applied for the case of $h=\rho$ in Proposition \ref{Prop Dn ell hPrince} yields, for general $n\geq 4$
$$
\ell^{\varpi_{n-1}+\varpi_n}=\ell^-_{\varpi_{n-1}+\varpi_n}(\varpi_1)=l(r_n...r_1)=n \;.
$$
Thus, back to $n=5$, we obtain $\ell^{\varpi_4+\varpi_5}=5$. We can conclude that $\ell_{{\bf D}_5}^{\rm sd}=\ell^{\varpi_1}=4$. This completes the proof.
\end{proof}

\subsection{Type ${\bf G}_2$}\label{Sect G2}

Let $\Delta$ be a root system of type ${\bf G}_2$. Then $w_0=-1$, so that all dominant weights are self-dual and $\ell_\Delta=\ell_\Delta^{\rm sd}$. The simple roots $\alpha_1,\alpha_2$ are numbered so that $\alpha_1$ is short, as in \cite{Bourbaki-Lie-2}.

\begin{prop}\label{Prop G2 ell}
The value of the invariant $\ell_\Delta$ for a root system of type ${\bf G}_2$ is $\ell_{{\bf G}_2}=3$. It is attained as $\ell^h$ for every dominant weight $h$. More precisely, we have
\begin{gather*}
\begin{array}{l}
\ell^{\varpi_1} = \ell^-_{\varpi_1}(\varpi_1)= \ell^-_{\varpi_1}(\lambda) = l(r_1r_2r_1) = 3 \;,\\
\ell^{\varpi_2} = \ell^-_{\varpi_2}(\varpi_2)= \ell^-_{\varpi_1}(\lambda) = l(r_2r_1r_2) = 3 \;,
\end{array}
\end{gather*}
for any $\lambda\in\Lambda^{++}$.
\end{prop}

\begin{proof}
Recall, \cite{Bourbaki-Lie-2}, that both fundamental weights of ${\bf G}_2$ are roots, $\varpi_1=2\alpha_1+\alpha_2$ is short and $\varpi_2=3\alpha_1+2\alpha_2$ is the highest root. Applying Example \ref{Exa shortroot} to this situation, we obtain
$$
\ell^{\varpi_1}={\rm hi}(\varpi_1)=3 \quad, \quad \ell^{\varpi_2}=\check{\rm hi}(\varpi_2^\vee)=3 \;.
$$
This proves $\ell_{{\bf G}_2}=3$. To see that $\ell^h=3$ for every strictly dominant $h\in\Lambda^+$, note that $r_1r_2r_1(\varpi_1)$ is a negative root, so $\langle r_1r_2r_1(\varpi_1),h\rangle<0$.
\end{proof}

\begin{rem}
In this case we can describe all possible values of $\ell^-_h(\lambda)$ for $h,\lambda\in\Lambda^+\setminus\{0\}$, which turn out to be, 3 and 4. Indeed, taking $h=\rho$ we compute:
$$
\ell^-_\rho(\lambda)= \begin{cases} 3=l(r_1r_2r_1)  &\;,\; {\rm if}\;\; \lambda_1>\lambda_2 \;,\\  
                                  4=l(r_1r_2r_1r_2)=l(r_2r_1r_2r_1) &\;,\; {\rm if}\;\; \lambda_1=\lambda_2 \;, \\
                                  3 = l(r_2r_1r_2) &\;,\; {\rm if}\;\; \lambda_1=\lambda_2 \;.\\
                    \end{cases}
$$
Now put $h=h_1\varpi_1+h_2\varpi_2$ and $\lambda=\lambda_1\varpi_1+\lambda_2\varpi_2$. The case $h_1=h_2$ is covered by the result for $\rho$. Assume $h_1> h_2$. Then
$$
\ell^-_h(\lambda) = \begin{cases} 3=l(r_1r_2r_1)  \;\;,\; {\rm if}\;\; \lambda_1\geq\lambda_2\;,\\  
                                  4=l(r_1r_2r_1r_2)\;\;,\; {\rm if}\;\; \lambda_1<\lambda_2 \;.
                    \end{cases}
$$
The case $h_1<h_2$ is analogous and the result is obtained by switching the indices 1,2 in the above formula.
\end{rem}

\subsection{Type ${\bf F}_4$}\label{Sect F4}

Let $\Delta$ be a root system of type ${\bf F}_4$. Then we have $w_0=-1$, so all weights are self-dual, $\Lambda^+=\Lambda^+_{\rm sd}$ and $\ell_\Delta=\ell_\Delta^{\rm sd}$. We adhere to the notation of \cite{Bourbaki-Lie-2}, where the simple roots $\alpha_1,\alpha_2,\alpha_3,\alpha_4$ are numbered so that $\alpha_1,\alpha_2$ are short and $\langle\alpha_2,\alpha_3\rangle\ne 0$.

\begin{prop}\label{Prop F4 ell}
For a root system $\Delta$ of type ${\bf F}_4$ the value of $\ell_\Delta$ is $8$ and is attained as $\ell^h$, where $h$ can be any fundamental weight:
\begin{gather*}
\ell_{{\bf F}_4}=8=\ell^{\varpi_j} \;,\;\; j=1,2,3,4.
\end{gather*}
\end{prop}

\begin{proof}
The result is deduced from the following table, where the box corresponding to a given pair $\lambda,h$ contains the number $l=\ell^-_h(\lambda)$ and, in some cases, a sequence $j_1...j_l$, encoding a minimal expression $w=r_{j_1}...r_{j_l}$ for a Weyl group element such that $\langle w\lambda,h\rangle<0$. The numbers $l$, where no $w$ is given are deduced from those with $w$ using the property $\ell^-_h(\lambda)=\ell^-_\lambda(h)$.\\

\begin{tabular}{|l|c|c|c|c|c|}
\hline
$\lambda\setminus h$ & $\varpi_1$ & $\varpi_2$ & $\varpi_3$ & $\varpi_4$ & $\rho$ \\
\hline
$\varpi_1$ & $8$ & $8$ & $9$ & $10$ & $8$ \\
& $_{12342321}$ & $_{23124321}$ & $_{321324321}$ & $_{4321324321}$ & $_{12324321}$ \\
\hline
$\varpi_2$ & $8$ & $10$ & $10$ & $9$ & $11$ \\
  & & $_{2342132312}$ & $_{3213432132}$ & $_{432132432}$ & $_{12342312312}$ \\
\hline
$\varpi_3$ & $9$ & $10$ & $10$ & $8$ & $10$ \\
  & & & $_{3231234323}$ & $_{43213243}$ & $_{1234321323}$ \\
\hline
$\varpi_4$ & $10$ & $9$ & $8$ & $8$ & $8$ \\
  & & & & $_{43213234}$ & $_{43213234}$ \\
\hline
$\rho$ & $8$ & $11$ & $10$ & $8$ & $11$ \\
  & & & & & $_{12321432132}$ \\
\hline
\end{tabular}

The results are verified by a computer program, \cite{Compute}.
\end{proof}

\subsection{Type ${\bf E}_6$}\label{Sect E6}

Here we let $\Delta$ have type ${\bf E}_6$, again with the standard notation of \cite{Bourbaki-Lie-2}. We have $w_0\ne-1$ and the duality between fundamental weight is given by $\varpi_1^*=\varpi_6$, $\varpi_2^*=\varpi_2$, $\varpi_3^*=\varpi_5$, $\varpi_4^*=\varpi_4$. Hence $\Lambda^+_{\rm sd}$ is generated by $\varpi_2,\varpi_4,\varpi_1+\varpi_6,\varpi_3+\varpi_5$.

\begin{prop}\label{Prop E6 ell and LSD}
For a root system $\Delta$ of type ${\bf E}_6$ the values of $\ell_\Delta$ and $\ell_\Delta^{\rm sd}$ are $5$ and $9$, respectively, attained as follows:
\begin{align*}
\ell_{{\bf E}_6}=&5=\ell^{\varpi_1}=\ell^-_{\varpi_6}(\varpi_1)=l(r_6r_5r_4r_3r_1) \;,\\
\ell_{{\bf E}_6}^{\rm sd}=&9=\ell^{\rho}=\ell^-_{\rho}(\varpi_1)=l(r_1r_3r_4r_2r_6r_5r_4r_3r_1) \;.\\
\end{align*}
\end{prop}

\begin{proof}
The result is deduced from the following table, where the box corresponding to a given pair $\lambda,h$ contains the number $l=\ell^-_h(\lambda)$ and, in some cases, a sequence $j_1...j_l$, encoding a minimal expression $w=r_{j_1}...r_{j_l}$ for a Weyl group element such that $\langle w\lambda,h\rangle<0$. The numbers $l$, where no $w$ is given are deduced from those with $w$ using the property $\ell^-_h(\lambda)=\ell^-_\lambda(h)$ and the symmetries of the root system. In the empty boxed the value of $l$ is at least $11$.\\

\begin{tabular}{|l|c|c|c|c|c|c|c|}
\hline
$\lambda\setminus h$ & $\varpi_1$ & $\varpi_2$ & $\varpi_3$ & $\varpi_4$ & $\varpi_5$ & $\varpi_6$ & $\rho$ \\
\hline
$\varpi_1$ & $8$ & $11$ & $7$ & $10$ & $8$ & $5$ & $9$ \\
  & ${_{13452431}}$ & ${_{24354265431}}$ & ${_{3425431}}$ & ${_{4354265431}}$ & ${_{54265431}}$ & ${_{65431}}$ & ${_{134265431}}$ \\
\hline
$\varpi_2$ & $11$ & $11$ & $11$ & $11$ & $11$ & $11$ & $11$ \\
  &  & ${_{24315436542}}$ & ${_{31425436542}}$ & ${_{42315436542}}$ & & & \\
\hline
$\varpi_3$ & $7$ & $11$ & & & $10$ & $8$ & \\
  &  & & & & ${_{5423165143}}$ & & \\
\hline
$\varpi_4$ & $10$ & $11$ & & & & $10$ & \\
\hline
$\varpi_5$ & $8$ & $11$ & $10$ & & & $7$ & \\
\hline
$\varpi_6$ & $5$ & $11$ & $9$ & $10$ & $7$ & $8$ & $9$ \\
\hline
$\rho$ & $9$ & $11$ & & & & $9$ & \\
  & ${_{134254316}}$ & & & & & ${_{654231435}}$ & \\
\hline
$\varpi_1+\varpi_6$ & $9$ & & & & & $9$ & \\
  & ${_{134254316}}$ & & & & & ${_{654321456}}$ & \\
\hline
$\varpi_3+\varpi_5$ & $9$ & & & & & $9$ & \\
  & ${_{134254365}}$ & & & & & ${_{654231435}}$ & \\
\hline
\end{tabular}

The results are verified by a computer program, \cite{Compute}.
\end{proof}

\subsection{Type ${\bf E}_7$}\label{Sect E7}

Let $\Delta$ have type ${\bf E}_7$, with the simple roots ordered as in \cite{Bourbaki-Lie-2}.

\begin{prop}\label{Prop E7 ell}
Let $\Delta$ be a root system of type ${\bf E}_7$. Then $\ell_\Delta=10$ and the value is attained for the pair $\lambda=h=\varpi_7$ for the Weyl group element $r_7r_6r_5r_4r_2r_3r_4r_5r_6r_7$.
\end{prop}

The result is verified using the computer program \cite{Compute}.\\

Let us also note that $\varpi_1$ is the highest root of ${\bf E}_7$, its height is $17$, and, from Example \ref{Exa shortroot} we deduce
$$
\ell^{\varpi_1}=17 \;.
$$
The value is attained as $\ell^-_{\varpi_1}(\varpi_j)$ each $j$ and a suitable Weyl group element.

\subsection{Type ${\bf E}_8$}\label{Sect E8}

Let $\Delta$ be a root system of type ${\bf E}_8$. We have unfortunately been unable to compute the exact value of $\ell_{{\bf E}_8}$. However, we supply lower and upper bounds. Again we follow the numbering of \cite{Bourbaki-Lie-2}.

\begin{prop}\label{Prop E8 ell bounds}
The following inequalities hold:
$$
7 = \ell_{{\bf D}_{8}}\leq \ell_{{\bf E}_8} \leq \ell^{\varpi_8} = 29 \;.
$$
\end{prop}

\begin{proof}
The upper bound is deduced from Example \ref{Exa shortroot}, since $\varpi_8$ is the highest root and $29$ is its height. It can also be verified that this length is attained as $\ell^-_{\varpi_8}(\lambda)$ for every nonzero $\lambda\in\Lambda^+$ and a suitable Weyl group element $w$. These $w$ are distinct for the distinct fundamental weights $\varpi_j$, say $w_j$, but each $w_j$ suits all strictly dominant $\lambda$.

The lower bound is deduced from the fact that ${\bf E}_8$ contains a subsystem of type ${\bf D}_8$, and the Weyl chambers can be chosen so that $\Lambda_{\RR}^+({\bf E}_8)\subset \Lambda_{\RR}^+({\bf D}_8)$. This implies the inequality $\ell_{{\bf D}_8}\leq \ell_{{\bf E}_8}$. The value $\ell_{{\bf D}_8}=7$ is obtained from \ref{Prop Dn ell and LSD}.
\end{proof}

\bibliographystyle{plain}

\small{

}

\vspace{0.4cm}

\begin{tabular}{ll}
Valdemar V. Tsanov & Yana Staneva \\
Ruhr-Universit\"at Bochum & University of Cologne \\
Fakult\"at f\"ur Mathematik, IB 3/101 $\qquad\qquad\quad$ & Mathematical Institute, Room 1.19\\ 
D-44780 Bochum, Germany. & 50931 Cologne, Germany.\\
Email: Valdemar.Tsanov@rub.de & Email: ystaneva@math.uni-koeln.de \\
\end{tabular}

\end{document}